\documentclass[smallcondensed,envcountsect,numbook]{svjour3}

\usepackage{csl}
\usepackage{subcaption}
\usepackage{graphicx}
\usepackage{amsmath}
\usepackage{amsfonts}
\usepackage{cleveref}
\usepackage{mleftright}
\usepackage{bbm}
\usepackage{siunitx}
\usepackage{physics}
\usepackage{mathtools}
\usepackage{url}
\usepackage{xparse}
\usepackage{pgfplots}
\usepackage{tikz}

\smartqed
\graphicspath{{./Figures/}}

\DeclareMathOperator{\diag}{diag}
\DeclarePairedDelimiter{\ceil}{\lceil}{\rceil}

\newenvironment{butchertableau}[2][1.35]
{\def\arraystretch{#1}\array{#2}}
{\endarray}

\NewDocumentCommand{\zeros}{m g}{\mathbf{0}_{{#1} \times \IfNoValueTF{#2}{#1}{#2}}}
\NewDocumentCommand{\eye}{m g}{\mathbf{I}_{{#1} \times \IfNoValueT{#2}{#1}}}
\newcommand{\one}{\mathbbm{1}}
\newcommand{\R}{\mathbbm{R}}
\newcommand{\Cplx}{\mathbbm{C}}

\newcommand{\instage}{Y}
\newcommand{\exstage}[1][n-1]{y^{[#1]}}

\newcommand{\A}{\mathbf{A}}
\newcommand{\B}{\mathbf{B}}
\newcommand{\U}{\mathbf{U}}
\newcommand{\V}{\mathbf{V}}
\renewcommand{\c}{\mathbf{c}}
\newcommand{\C}[1]{\mathbf{C}_{#1}}

\newcommand{\M}{\mathbf{M}}
\newcommand{\shiftmat}[1]{\mathbf{K}_{#1}}
\newcommand{\expshiftmat}[1]{\mathbf{E}_{#1}}
\newcommand{\hessmat}[1]{\widehat{\mathbf{E}}_{#1}}
\newcommand{\phishiftmat}[1]{\mathbf{F}_{#1}}

\renewcommand{\AE}{\mathbf{A}}
\newcommand{\BE}{\mathbf{B}}
\newcommand{\QE}{\mathbf{W}}
\newcommand{\qE}{w}
\newcommand{\ME}{\mathbf{M}}
\newcommand{\AI}{\widehat{\mathbf{A}}}
\newcommand{\BI}{\widehat{\mathbf{B}}}
\newcommand{\QI}{\widehat{\mathbf{W}}}
\newcommand{\qI}{\widehat{w}}
\newcommand{\MI}{\widehat{\mathbf{M}}}

\newcommand{\sRegionE}[1][\alpha]{\ensuremath{\mathcal{S}}}
\newcommand{\sRegionI}[1][\alpha]{\ensuremath{\widehat{\mathcal{S}}}}
\newcommand{\sRegionStiff}[1][\alpha]{\ensuremath{\widehat{\mathcal{S}}_{#1}}}
\newcommand{\sRegionNonstiff}[1][\alpha]{\ensuremath{\mathcal{S}_{#1}}}

\pgfplotsset{compat=1.14}

\pgfplotscreateplotcyclelist{mycolor}{%
	blue,every mark/.append style={fill=blue!80!black},mark=*\\%
	red,every mark/.append style={fill=red!80!black},mark=square*\\%
	teal,every mark/.append style={fill=teal!80!black},mark=triangle\\%
	black,mark=star\\%
	blue,every mark/.append style={fill=blue!80!black},mark=diamond*\\%
	red,densely dashed,every mark/.append style={solid,fill=red!80!black},mark=*\\%
	brown!60!black,densely dashed,every mark/.append style={
		solid,fill=brown!80!black},mark=square*\\%
	black,densely dashed,every mark/.append style={solid,fill=gray},mark=otimes*\\%
	blue,densely dashed,mark=star,every mark/.append style=solid\\%
	red,densely dashed,every mark/.append style={solid,fill=red!80!black},mark=diamond*\\%
}

\pgfplotsset{every axis/.append style={
		width=\linewidth,
		x label style={font=\small},
		y label style={font=\small},
		tick label style={font=\small},
		legend cell align={left},
		legend columns=1,
		legend style={at={(0.5,1.05)}, anchor=south, font=\tiny},
		cycle list name=mycolor,
		mark size=1.4pt,
        max space between ticks=20pt,
         try min ticks=3,
}}

\newif\ifreport
\reporttrue

\begin{document}

\ifreport
	\csltitle{Parallel implicit-explicit general linear methods}
	\cslauthor{Steven Roberts, Arash Sarshar, and Adrian Sandu}
	\cslemail{steven94@vt.edu, sarshar@vt.edu, sandu@cs.vt.edu}
	\cslreportnumber{12}
	\cslyear{19}
	\csltitlepage
\fi

\title{Parallel implicit-explicit general linear methods\thanks{This work was funded by awards NSF CCF–1613905, NSF ACI–1709727, AFOSR DDDAS FA9550-17-1-0015, and by the Computational Science Laboratory at Virginia Tech.}
}


\author{Steven Roberts \and Arash Sarshar \and Adrian Sandu}


\newcommand{\instituteinfo}{\csl{}, \cslinstitution{}, \csladdress{}}
\institute{S. Roberts (Corresponding author) \at
		\instituteinfo \\
		\email{steven94@vt.edu}
	\and
		A. Sarshar \at
		\instituteinfo \\
		\email{sarshar@vt.edu}
	\and
		A. Sandu \at
		\instituteinfo \\
		\email{sandu@cs.vt.edu}
}

\date{Received: date / Accepted: date}

\maketitle

\begin{abstract}
	High-order discretizations of partial differential equations (PDEs) necessitate high-order time integration schemes capable of handling both stiff and nonstiff operators in an efficient manner.  Implicit-explicit (IMEX) integration based on general linear methods (GLMs) offers an attractive solution due to their high stage and method order, as well as excellent stability properties.  The IMEX characteristic allows stiff terms to be treated implicitly and nonstiff terms to be efficiently integrated explicitly. This work develops two systematic approaches for the development of IMEX GLMs of arbitrary order with stages that can be solved in parallel.  The first approach is based on diagonally implicit multistage integration methods (DIMSIMs) of types 3 and 4.  The second is a parallel generalization of IMEX Euler and has the interesting feature that the linear stability is independent of the order of accuracy.  Numerical experiments confirm the theoretical rates of convergence and reveal that the new schemes are more efficient than serial IMEX GLMs and IMEX Runge--Kutta methods.
	
\end{abstract}

\section{Introduction}

In this work, we consider the autonomous, additively partitioned system of ordinary differential equations (ODEs)
\begin{equation} \label{eq:ode}
	y' = f(y) + g(y), \quad y(t_0) = y_0, \quad t_0 \le t \le t_{F},
\end{equation}
where $f$ is nonstiff, $g$ is stiff, and $y \in \R^d$.  Such systems frequently arise from applying the methods of lines to semidiscretize a partial differential equation (PDE).  For example, processes such as diffusion, advection, and reaction all have different stiffnesses, CFL conditions, and optimal integration schemes.  Implicit-explicit (IMEX) methods offer a specialized approach for solving \cref{eq:ode} by treating $f$ with an inexpensive explicit method and limiting the application of an implicit method, which is generally more expensive, to $g$.

The IMEX strategy has a relatively long history in the context of Runge--Kutta methods \cite{ascher1997implicit,Kennedy2003,pareschi2005implicit,boscarino2009class,Sandu_2010_extrapolatedIMEX} and linear multistep methods \cite{ascher1995implicit,frank1997stability,hundsdorfer2007imex}. Zhang, Zharovski, and Sandu proposed IMEX schemes based on two-step Runge--Kutta (TSRK) and General Linear Methods (GLM) \cite{Sandu_2014_IMEX-GLM,Sandu_2015_IMEX-TSRK,Sandu_2012_ICCS-IMEX} with further developments  reported in \cite{Sandu_2014_IMEX-RK,Sandu_2014_IMEX_GLM_Extrap,Sandu_2015_Stable_IMEX-GLM,Sandu_2016_highOrderIMEX-GLM,JACKIEWICZ2017,bras2017accurate,BRAS2018207,izzo2019transformed}. Similarly, Peer methods, a subclass of GLMs, have been utilized for IMEX integration in the literature such as \cite{lang2017extrapolation,schneider2018extrapolation,soleimani2018superconvergent,ditkowski2019imex}.

High-order IMEX GLMs do not face the stability barriers that constrain multistep counterparts and have much simpler order conditions than IMEX Runge--Kutta methods.  Moreover, they can attain high stage order making them resilient to the order reduction phenomena seen in very stiff problems and PDEs with time-dependent boundary conditions.

A major challenge when deriving high-order IMEX GLMs is ensuring the stability region is large enough to be competitive with IMEX Runge--Kutta schemes.  One can directly optimize for the area of the stability region under the constraints of the order conditions, but this is quite challenging as the objective and constraint functions are highly nonlinear and expensive to evaluate.  In addition, this optimization is not scalable, with sixth order appearing to be the highest order achieved with this strategy \cite{JACKIEWICZ2017}.

Parallelism for IMEX schemes is scarcely explored \cite{Connors2011,ditkowski2019imex}, but it is well-studied for traditional, unpartitioned GLMs \cite{butcher1993general,butcher1995parallel,butcher1997order,jackiewicz2009general}.  One step of a GLM is:
\begin{alignat*}{2}
	\instage_i &= h \sum_{j=1}^s a_{i,j} \left( f(\instage_j) + g(\instage_j) \right) + \sum_{j=1}^r u_{i,j} \, \exstage_j, & \qquad i &= 1, \dots, s, \\
	\exstage[n]_i &= h \sum_{j=1}^s b_{i,j} \left( f(\instage_j) + g(\instage_j) \right) + \sum_{j=1}^r v_{i,j} \, \exstage_j, & \qquad i &= 1, \dots, r.
\end{alignat*}
Methods are frequently categorized into one of four types to characterize suitability for stiff problems and parallelism \cite{butcher1993diagonally}.  Type 1 and 2 are serial and have the structure
\begin{equation} \label{eq:a_coeff}
	\A = \begin{bmatrix}
		\lambda \\
		a_{2,1} & \lambda \\
		\vdots & \vdots & \ddots \\
		a_{s,1} & a_{s,2} & \cdots & \lambda
	\end{bmatrix}.
\end{equation}
When $\lambda = 0$, the method is of type 1, and of type 2 for $\lambda > 0$.  Of interest to this paper are methods of type 3 and 4, which have $A = \lambda \, \eye{s}$ so that all internal stages are independent and can be computed in parallel.  Type 3 methods are explicit with $\lambda = 0$, while type 4 methods are implicit with $\lambda > 0$.

This work extends traditional, parallel GLMs into the IMEX setting and proposes two systematic approaches for designing stable methods of arbitrary order.  The first uses the popular DIMSIM framework for the base methods.  In particular, we use a family of type 4 methods proposed by Butcher \cite{butcher1997order} for the implicit base and show an explicit counterpart is uniquely determined.  This eliminates the need to perform a sophisticated optimization procedure to determine coefficients.  The second approach can be interpreted as a generalization of the simplest IMEX scheme: IMEX Euler.  It starts with an ensemble of states each approximating the ODE solution at different points in time.  In parallel, they are propagated one timestep forward using IMEX Euler, which is only first order accurate.  A new, highly-accurate ensemble of states is computed by taking linear combinations of the IMEX Euler solutions.  This scheme, which we call \textit{parallel ensemble IMEX Euler}, can be described in the framework of IMEX GLMs. Notably, it maintains the exact same stability region and roughly the same runtime in a parallel setting as IMEX Euler while achieving arbitrarily high orders of consistency.  Again, coefficients are determined uniquely, and we show they are very simple to compute using basic matrix operations.

To assess the quality of the two new families of parallel IMEX GLMs, we apply them to a PDE with time-dependent forcing and boundary conditions, as well as to a singularly perturbed PDE.  Convergence is verified as high as eighth order for these challenging problems which can cause order reduction for methods of low stage order.  For the performances tests, the parallel methods were run on several nodes in a cluster using MPI and compared to existing, high-quality, serial IMEX Runge--Kutta and IMEX GLMs run on a single node.  The best parallel methods could reach a desired solution accuracy approximately two to four times faster.

The structure of this paper is as follows.  \Cref{sec:background} reviews the formulation, order conditions, and stability analysis of IMEX GLMs.  This is then specialized in \cref{sec:parallel_imex_glm} for parallel IMEX GLMs.  \Cref{sec:parallel_imex_dimsim,sec:parallel_ensemble} present and analyze two new families of parallel IMEX GLMs.  The convergence and performance of these new schemes is compared to other IMEX GLMs and IMEX Runge--Kutta methods in \cref{sec:experiments}.  We summarize our findings and provide final remarks in \cref{sec:conclusion}.
\section{Background on IMEX GLMs}
\label{sec:background}
An IMEX GLM \cite{Sandu_2014_IMEX-GLM} computes $s$ internal and $r$ external stages using timestep $h$ according to:
\begin{subequations} \label{eq:imex_glm}
	\begin{alignat}{2}
		\instage_i &= h \sum_{j=1}^{i-1} a_{i,j} \, f(Y_j) + h \sum_{j=1}^{i} \widehat{a}_{i,j} \, g(Y_j)  + \sum_{j = 1}^r u_{i,j} \, \exstage_j, & \qquad i &= 1, \dots, s, \\
		\exstage[n]_i &= h \sum_{j=1}^s \left( b_{i,j} \, f(Y_j) + \widehat{b}_{i,j} \, g(Y_j) \right) + \sum_{j = 1}^r v_{i,j} \, \exstage_j, & \qquad i &= 1, \dots, r.
	\end{alignat}
\end{subequations}
Using the matrix notation for the coefficients
\begin{equation*}
\AE:= (a_{i,j}), \quad \BE:=(b_{i,j}), \quad \U:=(u_{i,j}), \quad	\AI:= (\hat{a}_{i,j}), \quad \BI:=(\hat{b}_{i,j}) \quad \V:=(v_{i,j}),
\end{equation*}
the IMEX GLM can be represented in the Butcher tableau
\begin{align}
	\begin{butchertableau}{c|c|c|c}
		\c & \AE & \AI & \U \\ \hline
		& \BE & \BI & \V
	\end{butchertableau}.
\end{align}

Assuming the incoming external stages to a step satisfy
\begin{equation} \label{eq:incoming_ex_stages}
	\begin{split}
		\exstage_i &= \qE_{i,0} \, y(t_{n-1}) +  \sum_{k=1}^{p} \qE_{i,k} \, h^k \, {\dv[k-1]{f(y(t))}{t}}(t_{n-1}) \\
		& \quad + \sum_{k=1}^{p} \qI_{i,k} \, h^k \, {\dv[k-1]{g(y(t))}{t}}(t_{n-1}) + \order{h^{p+1}}, \qquad i = 1, \dots, r,
	\end{split}
\end{equation}
an IMEX GLM is said to have \textit{stage order} $q$ if 
\begin{equation*}
	\instage_i  = y(t_{n-1} + c_i \, h) + \order{h^{q+1}}, \qquad i = 1, \dots, s,
\end{equation*} 
and \textit{order} $p$ if 
\begin{align*}
	 \exstage[n]_i &= \qE_{i,0} \, y(t_{n}) +  \sum_{k=1}^{p} \qE_{i,k} \, h^k \, {\dv[k-1]{f(y(t))}{t}}(t_n) \\
	 & \quad + \sum_{k=1}^{p} \qI_{i,k} \, h^k \, {\dv[k-1]{g(y(t))}{t}}(t_n) + \order{h^{p+1}}, \qquad i = 1, \dots, r.
\end{align*}
The Taylor series weights for the external stages are also described in the matrix form
\begin{equation}
	\QE = (\qE_{i,j}) \in \R^{r \times (p+1)}, \qquad \QI = (\qI_{i,j}) \in \R^{r \times (p+1)},
\end{equation}
with $w_{i,0} = \widehat{w}_{i,0}$ for $i = 1, \dots, r$.

The order conditions for IMEX GLMs are discussed in detail in \cite{Sandu_2014_IMEX-GLM}.  Notably, a preconsistent IMEX GLM has order $p$ and stage order $q \in \{p, p-1\}$ if and only if the base methods have order $p$ and stage order $q \in \{p, p-1\}$.  Here, we present the order conditions in a compact matrix form.  First, we define the Toeplitz matrices
\begin{equation*}
	\shiftmat{n} = \begin{bmatrix}
		0 & 1 \\
		& 0 & 1 \\
		& & \ddots & \ddots \\
		& & & 0 & 1 \\
		& & & & 0
	\end{bmatrix} \in \R^{n \times n},
	\qquad
	\expshiftmat{n} = \exp(\shiftmat{n}) = \begin{bmatrix}
		1 & 1 & \frac{1}{2} & \cdots & \frac{1}{(n-1)!} \\
		& 1 & 1 & \cdots & \frac{1}{(n-2)!} \\
		& & \ddots & \ddots & \vdots \\
		& & & 1 & 1 \\
		& & & & 1
	\end{bmatrix} \in \R^{n \times n},
\end{equation*}
and scaled Vandermonde matrix
\begin{equation}
	\C{n} = \begin{bmatrix}
		\one_s & \c & \frac{\c^2}{2} & \cdots & \frac{\c^{n-1}}{(n-1)!}
	\end{bmatrix} \in \R^{s \times n}.
\end{equation}
Powers of a vector are understood to be component-wise, and $\one_s$ represents the vector of ones of dimension $s$.

\begin{theorem}[Compact IMEX GLM order conditions \cite{Sandu_2014_IMEX-GLM}] \label{thm:imex_glm_oc}
	Assume $\exstage$ satisfies \cref{eq:incoming_ex_stages}.  The IMEX GLM \cref{eq:imex_glm} has order $p$ and stage order $q \in \{p, p-1\}$  if and only if
	\begin{subequations} \label{eq:imex_glm_oc}
		\begin{align}
			\label{eq:imex_glm_oc:int_e}
			\C{q+1} - \AE \, \C{q+1} \, \shiftmat{q+1} - \U \, \QE_{:,0:q} = \zeros{s}{(q+1)}, \\
			\label{eq:imex_glm_oc:int_i}
			\C{q+1} - \AI \, \C{q+1} \, \shiftmat{q+1} - \U \, \QI_{:,0:q} = \zeros{s}{(q+1)}, \\
			\label{eq:imex_glm_oc:ext_e}
			\QE \, \expshiftmat{p+1} - \BE \, \C{p+1} \, \shiftmat{p+1} - \V \, \QE = \zeros{r}{(p+1)}, \\
			\label{eq:imex_glm_oc:ext_i}
			\QI \, \expshiftmat{p+1} - \BI \, \C{p+1} \, \shiftmat{p+1} - \V \, \QI = \zeros{r}{(p+1)},
		\end{align}
		where $\QE_{:,0:q}$ is the first $q+1$ columns of $\QE$, and $\QI_{:,0:q}$ is defined analogously.
	\end{subequations}
\end{theorem}

\begin{remark}
	The first column in each of the matrix conditions in \cref{eq:imex_glm_oc} corresponds to a preconsistency condition.  
\end{remark}

\subsection{Linear stability of IMEX GLMs}

The standard test problem used to analyze the linear stability of an IMEX method is the partitioned problem
\begin{equation} \label{eq:linear_ode}
y' = \xi \, y + \widehat{\xi} \, y,
\end{equation}
where $\xi \, y$ is considered nonstiff and $\widehat{\xi} \, y$ is considered stiff.  Applying the IMEX GLM \cref{eq:imex_glm} to \cref{eq:linear_ode} yields the stability matrix
\begin{equation*} \label{eq:imex_glm_stability}
	\begin{split}
		\exstage[n] &= \M(w,\widehat{w}) \, \exstage[n-1], \\
		\M(w,\widehat{w}) &= \V + \left( w \, \BE + \widehat{w} \, \BI \right) \left(\eye{s} - w \, \AE - \widehat{w} \, \AI \right)^{-1} \U,
	\end{split}
\end{equation*}
where $w=h \, \xi$ and $\widehat{w}=h \, \widehat{\xi}$.  The set of $(w, \widehat{w}) \in \Cplx \times \Cplx$ for which $\M(w,\widehat{w})$ is power bounded, and thus the IMEX GLM is stable, is a four dimensional region that can be difficult to analyze and visualize.  Following \cite{Sandu_2014_IMEX-GLM}, we also consider the simpler stability regions
\begin{subequations} \label{eq:imex_glm_stability_regions}
	\begin{align}
		\label{eq:imex_glm_stability_regions:stiff}
		\sRegionStiff &= \left\{ \widehat{w} \in \sRegionI \, : \, \abs{\Im(\widehat{w})} < \tan(\alpha) \, \abs{\Re(\widehat{w})} \right\}, \\
		\label{eq:imex_glm_stability_regions:nonstiff}
		\sRegionNonstiff &= \left\{ w \in \sRegionE \, : \, \M(w,\widehat{w}) \text{ power bounded } \forall \widehat{w} \in \sRegionStiff \right\},
	\end{align}
\end{subequations}
where \sRegionE{} and \sRegionI{} are the stability regions of the explicit and implicit base methods, respectively.  \Cref{eq:imex_glm_stability_regions:stiff} is referred to as the \textit{desired stiff stability region} and \cref{eq:imex_glm_stability_regions:nonstiff} as the \textit{constrained nonstiff stability region}.

\section{Parallel IMEX GLMs}
\label{sec:parallel_imex_glm}

An IMEX GLM formed by pairing a type 3 GLM with a type 4 GLM has stages of the form
\begin{subequations} \label{eq:parallel_imex_glm}
	\begin{alignat}{2}
		\instage_i &= h \, \lambda \, g(Y_i) + \sum_{j = 1}^r u_{i,j} \, \exstage_j, & \qquad i &= 1, \dots, s, \\
		\exstage[n]_i &= h \sum_{j=1}^s \left( b_{i,j} \, f(Y_j) + \widehat{b}_{i,j} \, g(Y_j) \right) + \sum_{j = 1}^r v_{i,j} \, \exstage_j, & \qquad i &= 1, \dots, r.
	\end{alignat}
\end{subequations}
The only shared dependencies among the internal stages are the previously computed external stages $\exstage_j$.  This allows the IMEX method to inherit the parallelism of the base methods.

The tableau for a parallel IMEX GLM is of the form
\begin{equation} \label{eq:parallel_imex_glm_tableau}
	\begin{butchertableau}{c|c|c|c}
		\c & \zeros{s} & \lambda \, \eye{s} & \U \\ \hline
		& \BE & \BI & \V
	\end{butchertableau}.
\end{equation}
We note that one could more generally define $\AI = \diag{(\lambda_1, \dots, \lambda_s)}$, however, this introduces additional complexity and degrees of freedom that are not needed for the purposes of this paper.

\subsection{Simplified order conditions}
\label{sec:parallel_imex_glm:oc}

In this paper, we will consider methods with $p=q=r=s$, distinct $\c$ values (nonconfluent method), and an invertible $\U$.  By transforming the base methods into an equivalent formulation, we can then assume without loss of generality that $\U = \eye{s}$.  With these assumptions, we start by determining the structure of the external stage weights $\QE$ and $\QI$.

\begin{lemma} \label{lem:parallel_imex_glm_q}
	For a parallel IMEX GLM with $\U = \eye{s}$ and $p=q$, the internal stage order conditions \cref{eq:imex_glm_oc:int_e,eq:imex_glm_oc:int_i} are equivalent to
	\begin{equation} \label{eq:parallel_imex_glm_q}
		\QE = \C{p+1},
		\quad \text{and} \quad
		\QI = \C{p+1} - \lambda \, \C{p+1} \, \shiftmat{p+1},
	\end{equation}
	respectively.
\end{lemma}

\begin{proof}
	This follows directly from substituting $\AE = \zeros{s}$, $\AI = \lambda \, \eye{s}$, and $\U = \eye{s}$ into \cref{eq:imex_glm_oc:int_e,eq:imex_glm_oc:int_i}.
\end{proof}

Our main theoretical result on parallel IMEX GLMs is presented in \cref{thm:parallel_imex_glm_b} and provides a practical strategy for method derivation.

\begin{theorem}[Parallel IMEX GLM order conditions] \label{thm:parallel_imex_glm_b}
	Consider a nonconfluent parallel IMEX GLM with $\U = \eye{s}$.  All of the following are equivalent:
	\begin{enumerate}
		\item The method satisfies $p=q=r=s$.
		\item The explicit base method satisfies $p=q=r=s$ and
		\begin{subequations}
			\begin{align}
				\label{eq:parallel_imex_glm_b:ex:q}
				\QI &= \C{s+1} - \lambda \, \C{s+1} \, \shiftmat{s+1}, \\
				\label{eq:parallel_imex_glm_b:ex:b}
				\BI &= \BE - \lambda \, \C{s} \, \expshiftmat{s} \, \C{s}^{-1} + \lambda \, \V.
			\end{align}
		\end{subequations}
		\item The implicit base method satisfies $p=q=r=s$ and
		\begin{subequations}
			\begin{align}
				\QE &= \C{s+1}, \\
				\BE &= \BI + \lambda \, \C{s} \, \expshiftmat{s} \, \C{s}^{-1} - \lambda \, \V.
			\end{align}
		\end{subequations}
	\end{enumerate}
\end{theorem}

\begin{remark}
	With \cref{thm:parallel_imex_glm_b}, once the implicit base method has been chosen, \textit{all} coefficients for the explicit counterpart are uniquely determined by the order conditions.  Conversely, if the explicit base is fixed, then all implicit method coefficients are uniquely determined, but parameterized by $\lambda$.
\end{remark}

\begin{proof}
	To start, we will show the first statement of \cref{thm:parallel_imex_glm_b} is equivalent to the second.  Assume a nonconfluent parallel IMEX GLM with $\U = \eye{s}$ has $p=q=r=s$.  By \cref{thm:imex_glm_oc}, the explicit (and implicit) base method also has $p=q=r=s$ and satisfies the order conditions in \cref{eq:imex_glm_oc}.  Further, by \cref{lem:parallel_imex_glm_q}, \cref{eq:parallel_imex_glm_b:ex:q} holds.  Subtracting \cref{eq:imex_glm_oc:ext_i} from \cref{eq:imex_glm_oc:ext_e} gives
	\begin{equation} \label{eq:imex_glm_ext_diff}
		\lambda \, \C{s+1} \, \shiftmat{s+1} \, \expshiftmat{s+1} + \left( \BI - \BE \right) \C{s+1} \, \shiftmat{s+1} - \lambda \, \V \, \C{s+1} \, \shiftmat{s+1} = \zeros{s}{(s+1)}.
	\end{equation}
	The three terms summed on the left-hand side of \cref{eq:imex_glm_ext_diff} have zeros in the leftmost column.  Removing this yields the following equivalent statement:
	\begin{equation*}
		\lambda \, \C{s} \, \expshiftmat{s} + \left( \BI - \BE \right) \C{s} - \lambda \, \V \, \C{s} = \zeros{s}.
	\end{equation*}
	A bit of algebraic manipulation recovers the desired result of \cref{eq:parallel_imex_glm_b:ex:b}.
	
	Now assume a nonconfluent parallel IMEX GLM with $\U = \eye{s}$ satisfies the properties of the second statement of \cref{thm:parallel_imex_glm_b}.  Condition \cref{eq:parallel_imex_glm_b:ex:q} ensures the implicit method has stage order $q$, and \cref{eq:parallel_imex_glm_b:ex:b} ensure its has order $p$:
	\begin{align*}
		& \quad \QI \, \expshiftmat{s+1} - \BI \, \C{s+1} \, \shiftmat{s+1} - \V \, \QI \\
		&= \QI \, \expshiftmat{s+1} - \left( \BE - \lambda \, \C{s} \, \expshiftmat{s} \, \C{s}^{-1} + \lambda \, \V \right) \, \C{s+1} \, \shiftmat{s+1} - \V \, \QI \\
		&= \left( \QI - \lambda \, \C{s+1} \, \shiftmat{s+1} \right) \, \expshiftmat{s+1} - \BE \, \C{s+1} \, \shiftmat{s+1} - \V \left( \QI - \lambda \, \C{s+1} \, \shiftmat{s+1} \right) \\
		&= \QE \, \expshiftmat{s+1} - \BE \, \C{s+1} \, \shiftmat{s+1} - \V \, \QE \\
		&= \zeros{s}{(s+1)}.
	\end{align*}
	Now both base methods have $p=q=r=s$, so by \cref{thm:imex_glm_oc}, the combined IMEX scheme also has $p=q=r=s$.
	
	The process to show statement one is equivalent to statement three, thus completing the proof, follows nearly identical steps, and is therefore omitted.
\end{proof}

\subsection{Stability}

Applying parallel IMEX GLMs to linear stability test \cref{eq:linear_ode} gives
\begin{subequations} \label{eq:parallel_imex_glm_stability}
	\begin{align}
		\M(w, \widehat{w})
		&= \V + \frac{w}{1 - \lambda \, \widehat{w}} \, \BE \, \U + \frac{\widehat{w}}{1 - \lambda \, \widehat{w}} \, \BI \, \U \\
		&= \ME \mleft( \frac{w}{1 - \lambda \, \widehat{w}} \mright) + \MI(\widehat{w}) - \V,
	\end{align}
\end{subequations}
where $\ME(w)$ and $\MI(\widehat{w})$ are the stability matrices of the explicit and implicit base methods, respectively.  When the implicit partition becomes infinitely stiff,
\begin{equation*}
	\M(w, \infty) = \MI(\infty) = \V - \frac{1}{\lambda} \BI \, \U.
\end{equation*}
Stability matrices evaluated at $\infty$ are understood to be the value in the limit.

\subsection{Starting procedure}

The starting procedure for nontrivial IMEX GLMs is more complex than traditional GLMs because the external stages for IMEX GLMs weight time derivatives of $f$ and $g$ differently.  When computing $\exstage[0]$, the high order time derivatives are usually not readily available, but can be approximated by finite differences \cite{Califano2017,Sandu_2014_IMEX-GLM}.  A one-step method can be used to get very accurate approximations to $y$, and consequently $f$ and $g$, at a grid of time points around $t_0$ to construct these finite difference approximations.  While this generic approach is applicable to parallel IMEX GLMs, we also describe a specialized strategy that is simpler and more accurate.

Based on the $\QE$ and $\QI$ weights derived in \cref{eq:parallel_imex_glm_q},
\begin{equation} \label{eq:alternate_external_stages}
	\begin{split}
		\exstage[0]_i &= y(t_0) +  \sum_{k=1}^{p} \frac{c_i^k}{k!} \, h^k \, \dv[k-1]{f(y(t))}{t} (t_0) \\
		& \quad + \sum_{k=1}^{p} \left( \frac{c_i^k}{k!} - \frac{\lambda \, c_i^{k-1}}{(k-1)!} \right) h^k \, \dv[k-1]{g(y(t))}{t} (t_0) +  \order{h^{p+1}} \\
		&= y(t_0 + h \, c_i) - h \, \lambda \, g(y(t_0 + h \, c_i)) + \order{h^{p+1}}.
	\end{split}
\end{equation}
Now, a one-step method can be used to get approximations to $y$ and $g$ at  times $t_0 + h \, c_i$ to compute $\exstage[0]$.  This eliminates the need to use finite differences and eliminates the error associated with them.  Note that negative abscissae would require integrating backwards in time.  Although the interval of integration may be quite short, this could still lead to stability issues, and is easily remedied.   If $c_\text{min}$ is the smallest abscissa, then the one-step method can produce an approximation to $\exstage[\ell]$, where $\ell = \ceil{-c_\text{min}}$, instead of $\exstage[0]$.  Note, $t_\ell + c_i \, h \ge t_0$, and the IMEX GLM will start with $\exstage[\ell]$ to compute $\exstage[\ell+1]$ and so on.

\subsection{Ending procedure}

We will consider the ending procedure for an IMEX GLM to be of the form
\begin{equation}
	y(t_n) \approx h \sum_{j=1}^s \left( \beta_{j} \, f(Y_j) + \widehat{\beta}_{j} \, g(Y_j) \right) + \sum_{j = 1}^r \gamma_{j} \, \exstage_j.
\end{equation}
Frequently, IMEX GLMs have the last abscissa set to $1$, which allows for a particularly simple ending procedure for high stage order methods.  The final internal stage $Y_s$ can be used as an $\order{h^{\min(p, q+1)}}$ accurate approximation to $y(t_n)$.  One can easily verify that the coefficients for such an ending procedure are
\begin{equation} \label{eq:basic_ending_procedure}
	\beta^T = e_s^T \, \AE,
	\quad
	\widehat{\beta}^T = e_s^T \, \AI,
	\quad
	\gamma^T = e_s^T \, \U,
\end{equation}
where $e_i$ is the $i$-th column of $\eye{s}$. Indeed, all parallel IMEX GLMs tested in this paper have $c_s=1$, however, we present an alternative strategy to approximate $y(t_n)$.  Suppose a parallel IMEX GLM has $c_i=0$ for some $i \in \left\{ 1, \dots, s - 1 \right\}$ and $c_s=1$.  Then based on the relation in \cref{eq:alternate_external_stages}, we have that
\begin{align*}
	&\quad h \sum_{j=1}^s \left( b_{i,j} \, f(Y_j) + \widehat{b}_{i,j} \, g(Y_j) \right) + h \, \lambda \, g(\instage_s) + \sum_{j = 1}^r v_{i,j} \, \exstage_j \\
	&= \exstage[n]_i + h \, \lambda \, g(\instage_s) \\
	&= y(t_n) + \order{h^{\min(p, q+1)}}.
\end{align*}
This ending procedure has the coefficients
\begin{equation} \label{eq:alternate_ending_procedure}
	\beta^T = e_i^T \, \BE,
	\quad
	\widehat{\beta}^T = e_i^T \, \BI + \lambda \, e_s^T,
	\quad
	\gamma^T = e_i^T \, \V.
\end{equation}

For the parallel ensemble IMEX Euler methods of \cref{sec:parallel_ensemble}, numerical tests revealed this new ending procedure is substantially more accurate.  For the parallel IMEX DIMSIMs, the coefficients in \cref{eq:basic_ending_procedure,eq:alternate_ending_procedure} gave similar results in tests as the accumulated global error dominated the local truncation error of the ending procedure.
\section{Parallel IMEX DIMSIMs}
\label{sec:parallel_imex_dimsim}

Diagonally-implicit multi-stage integration methods (DIMSIMs) have become a popular choice of base method to build high-order IMEX GLMs.  IMEX DIMSIMs are characterized by the following structural assumptions:
\begin{enumerate}
	\item $\AE$ is strictly lower triangular, and $\AI$ is lower triangular with the same element $\lambda$ on the diagonal as in \cref{eq:a_coeff}.
	\item $\V$ is rank one with the single nonzero eigenvalue equal to one to ensure preconsistency.
	\item $q \in \{p, p-1\}$ and $r \in \{s, s+1\}$.
\end{enumerate}

Based on \cref{thm:imex_glm_oc}, to build a parallel IMEX DIMSIM with $p=q=r=s$ we only need to choose one of the base methods and the rest of the coefficients will follow.  If we start by picking an explicit base, it may be difficult to ensure the resulting implicit method has acceptable stability properties, ideally L-stability.  Instead, we start by picking a stable, type 4 DIMSIM for the implicit base method.

In \cite{butcher1993general,butcher1997order}, Butcher developed a systematic approach to construct DIMSIMs of type 4 with ``perfect damping at infinity.''  One of his primary results is presented in \cref{thm:dimsim_coeffs}.

\begin{theorem}[Type 4 DIMSIM coefficients {\cite[Theorem 4.1]{butcher1997order}}] \label{thm:dimsim_coeffs}
	For the type 4 DIMSIM
	\begin{equation*}
		\begin{butchertableau}{c|c|c}
			\c & \lambda \, \eye{s} &  \eye{s} \\ \hline
			& \BI & \V
		\end{butchertableau}
	\end{equation*}
	with $p=q=r=s$ and $\V \, \one_s = \one_s$, the transformed coefficients
	\begin{equation*}
		\overline{\B} = \mathbf{T}^{-1} \, \BI \, \mathbf{T},
		\qquad
		\overline{\V} = \mathbf{T}^{-1} \, \V \, \mathbf{T},
	\end{equation*}
	satisfy
	\begin{subequations}
		\begin{align}
			\label{eq:dimsim_coeffs:v}
			\overline{\V} e_1 &= e_1, \\
			\label{eq:dimsim_coeffs:b}
			\overline{\B} &= \hessmat{s} - \lambda \, \expshiftmat{s} + \overline{\V} \left(\lambda \, \eye{s} - \shiftmat{s}^T \right),
		\end{align}
	\end{subequations}
	where 
	\begin{equation*}
		\mathbf{T} = \begin{bmatrix}
			P^{(s)}(c_1) & P^{(s-1)}(c_1) & \cdots & P'(c_1) \\
			P^{(s)}(c_2) & P^{(s-1)}(c_2) & \cdots & P'(c_2) \\
			\vdots & \vdots & \ddots & \vdots \\
			P^{(s)}(c_s) & P^{(s-1)}(c_s) & \cdots & P'(c_s) \\
		\end{bmatrix}, \qquad
		P(x) = \frac{1}{s!} \prod_{i=1}^s (x-c_i),
	\end{equation*}
	and
	\begin{equation*}
		\hessmat{n} = \begin{bmatrix}
			1 & \frac{1}{2} & \frac{1}{6} & \cdots & \frac{1}{(n-1)!} & \frac{1}{n!} \\
			1 & 1 & \frac{1}{2} & \cdots & \frac{1}{(n-2)!} & \frac{1}{(n-1)!} \\
			0 & 1 & 1 & \cdots & \frac{1}{(n-3)!} & \frac{1}{(n-2)!} \\
			\vdots & \vdots & \vdots & & \vdots & \vdots \\
			0 & 0 & 0 & \cdots & 1 & \frac{1}{2} \\
			0 & 0 & 0 & \cdots & 1 & 1
		\end{bmatrix} \in \R^{n \times n}.
	\end{equation*}
\end{theorem}

\Cref{thm:dimsim_coeffs} fully determines the $\BI$ coefficient for a type 4 DIMSIM, but $\c$, $\lambda$ and most of $\V$ remain undetermined.  Fortunately, this offers sufficient degrees of freedom to ensure $\MI(\infty)$ is nilpotent.  In \cite[Theorem 5.1]{butcher1997order}, Butcher proves $\lambda$ must be a solution to
\begin{subequations} \label{eq:dimsim_stability_coeffs}
	\begin{equation} \label{eq:dimsim_stability_coeffs:lambda}
		L'_{s+1} \mleft( \frac{s+1}{\lambda} \mright) = 0,
	\end{equation}
	and
	\begin{equation} \label{eq:dimsim_stability_coeffs:v}
		\overline{\V} = \begin{bmatrix}
			v_1 & v_2 & \cdots & v_s \\
			0 & 0 & \cdots & 0 \\
			\vdots & \vdots & \ddots & \vdots \\
			0 & 0 & \cdots & 0
		\end{bmatrix}, \qquad
		v_i = (-1)^{s+1} \frac{s-i+2}{s+1} \lambda^{i-1} L_{s+1}^{(s-i+2)} \mleft( \frac{s+1}{\lambda} \mright).
	\end{equation}
\end{subequations}
Here, $L_n(x) = \sum_{i=0}^n \binom{n}{i} (-x)^i / i!$ is the Laguerre polynomial and $L_n^{(m)}(x)$ is its $m$-th derivative.

With the implicit base method determined, we now turn to the explicit method.  Indeed, \cref{thm:imex_glm_oc} could be applied to recover $\BE$, but \cref{thm:dimsim_coeffs} provides a more direct approach.  \Cref{eq:dimsim_coeffs:b}, which is normally used for type 4 methods, remains valid when $\lambda=0$, and \cref{eq:dimsim_coeffs:v} is fulfilled because the implicit and explicit base methods share $\V$.

In summary, the coefficients for a parallel IMEX DIMSIM with $p=q=r=s$ are given by
\begin{alignat*}{2}
	\AE &= \zeros{s}, \qquad &
	\BE &= \mathbf{T} \left( \hessmat{s} - \lambda \, \expshiftmat{s} + \overline{\V} \left(\lambda \, \eye{s} - \shiftmat{s}^T \right) \right) \mathbf{T}^{-1}, \\
	\AI &= \lambda \, \eye{s}, \qquad &
	\BI &= \mathbf{T} \left( \hessmat{s} - \overline{\V} \, \shiftmat{s}^T \right) \mathbf{T}^{-1}, \\
	\U &= \eye{s}, \qquad &
	\V &= \mathbf{T} \, \overline{\V} \, \mathbf{T}^{-1},
\end{alignat*}
with $\c$ remaining as free parameters.  The two most ``natural'' and frequently used choices are $\c = [0,1/(s-1),2/(s-2),\dots,1]^T$ and $\c = [2-s,1-s,\dots,1]^T$.  This presents a tradeoff where the first option has smaller local truncation errors, but the second option results  in  coefficients that grow slower with order, thus reducing the accumulation of finite precision cancellation errors.  \Cref{tab:imex_dimsim_max_coeff} presents the magnitude of these largest coefficients for both strategies.

\begin{table}
	\centering
	\begin{tabular}{c|c|c|c}
		Method order & $\lambda$ & $c_i = \frac{i-1}{s-1}$ & $c_i = 1 - s + i$ \\ \hline
		2 & 0.633975 & 1.38 & 1.38 \\
		3 & 1.21014 & 20.38 & 7.31 \\
		4 & 0.872421 & 90.86 & 7.07 \\
		5 & 1.30128 & 5885.22 & 29.74 \\
		6 & 1.80569 & 933038.32 & 368.93 \\
		7 & 1.35220 & 10318974.86 & 303.07 \\
		8 & 1.73680 & 2557191349.96 & 3534.00 \\
		9 & 1.38470 & 41543982719.05 & 2907.22 \\
		10 & 1.69561 & 14146161438042.40 & 41813.39
	\end{tabular}
	\caption{Approximate values for the largest coefficient in absolute value from $\BE$, $\BI$, and $\V$ for parallel IMEX DIMSIMs of orders two to ten.} \label{tab:imex_dimsim_max_coeff}
\end{table}

Before proceeding to the stability analysis, we present two examples of parallel IMEX DIMSIMs.  A second order method has the tableau
\begin{equation*}
	\begin{butchertableau}{c|cc|cc|cc}
		0 & 0 & 0 & \lambda  & 0 & 1 & 0 \\
		1 & 0 & 0 & 0 & \lambda  & 0 & 1 \\
		\hline
		\text{} & \frac{4 \lambda -3}{4} & \frac{4 \lambda -3}{4} & \frac{(2 \lambda +1) (4 \lambda -3)}{4} & \frac{-8 \lambda ^2+10 \lambda -3}{4} & \frac{4 \lambda -3}{2} & \frac{5-4 \lambda }{2} \\
		\text{} & \frac{4 \lambda -5}{4} & \frac{4 \lambda +3}{4} & \frac{8 \lambda ^2+2 \lambda -5}{4} & \frac{-8 \lambda ^2+6 \lambda +3}{4} & \frac{4 \lambda -3}{2} & \frac{5-4 \lambda }{2} \\
	\end{butchertableau},
\end{equation*}
where $\lambda = (3 - \sqrt{3}) / 2$.  In a more compact form, a third order method has the coefficients
\begin{align*}
	\c &= \begin{bmatrix}
		0 & \frac{1}{2} & 1
	\end{bmatrix}^T, \\
	\BE &= \begin{bmatrix}
		\frac{6 \lambda ^2-15 \lambda +7}{2} & \frac{6 \lambda -5}{3} & -\frac{(3 \lambda -2) (6 \lambda -13)}{6} \\
		\frac{72 \lambda ^2-180 \lambda +89}{24} & \frac{6 \lambda -7}{3} & \frac{-24 \lambda ^2+68 \lambda -27}{8} \\
		\frac{(3 \lambda -4) (6 \lambda -7)}{6} & 2 \lambda -5 & \frac{-18 \lambda ^2+51 \lambda -7}{6} \\
	\end{bmatrix}, \\
	\BI &= \begin{bmatrix}
		\frac{72 \lambda ^3-156 \lambda ^2+34 \lambda +21}{6} & \frac{-72 \lambda ^3+192 \lambda ^2-88 \lambda -5}{3} & \frac{36 \lambda ^3-114 \lambda ^2+80 \lambda -13}{3} \\
		\frac{288 \lambda ^3-624 \lambda ^2+112 \lambda +89}{24} & \frac{-72 \lambda ^3+192 \lambda ^2-79 \lambda -7}{3} & \frac{288 \lambda ^3-912 \lambda ^2+592 \lambda -81}{24} \\
		\frac{2 \left(18 \lambda ^3-39 \lambda ^2+4 \lambda +7\right)}{3} & \frac{-72 \lambda ^3+192 \lambda ^2-64 \lambda -15}{3} & \frac{72 \lambda ^3-228 \lambda ^2+130 \lambda -7}{6} \\
	\end{bmatrix}, \\
	\V &= \begin{bmatrix}
		1 \\ 1 \\ 1
	\end{bmatrix}
	\begin{bmatrix}
		\frac{72 \lambda ^2-174 \lambda +79}{6} & -\frac{2 \left(36 \lambda ^2-96 \lambda +47\right)}{3} & \frac{72 \lambda ^2-210 \lambda +115}{6} \\
	\end{bmatrix}, \\
	\lambda &= \frac{2 \cos \mleft(\frac{\pi }{18}\mright) \sec \mleft(\frac{\pi }{9}\mright)}{\sqrt{3}} \approx 1.210138312730603.
\end{align*}

\subsection{Stability}

While \cref{eq:dimsim_stability_coeffs} ensures $\rho( \MI(\infty) ) = 0$, it is not a sufficient condition for L-stability of a type 4 DIMSIM.  In \cite{butcher1997order}, appropriate values of $\lambda$ for L-stability are provided for orders two to ten, excluding nine.  If the weaker condition of L$(\alpha)$-stability is acceptable, smaller values of $\lambda$ may be used as well.

With $\c$ available as free parameters, it is natural to see if they can be used to optimize the stability of parallel IMEX DIMSIMs.  It is easy to verify that stability is, in fact, independent of $\c$:
\begin{align*}
	\mathbf{T}^{-1} \, \M(w, \widehat{w}) \, \mathbf{T} &= \overline{\V} + \frac{w}{1 - \lambda \, \widehat{w}} \left( \hessmat{s} - \overline{\V} \, \shiftmat{s}^T \right) \\
	& \quad + \frac{\widehat{w}}{1 - \lambda \, \widehat{w}} \left( \hessmat{s} - \lambda \, \expshiftmat{s} + \overline{\V} \left(\lambda \, \eye{s} - \shiftmat{s}^T \right) \right).
\end{align*}
The stability matrix is similar to a matrix completely independent of $\c$, thus $\c$ has no effect on the power boundedness of $\M$.

Plots of the constrained nonstiff stability region for several methods appear in \cref{fig:parallel_imex_dimsim_stability}.  Roughly speaking, the area of the stability region shrinks as the order increases.  Further, the smaller values of $\lambda$ satisfying \cref{eq:dimsim_stability_coeffs:lambda} tend to provide larger stability regions for a fixed order.

\begin{figure}[ht!]
	\centering
	\begin{subfigure}[b]{.45\linewidth}
		\includegraphics[width=\linewidth]{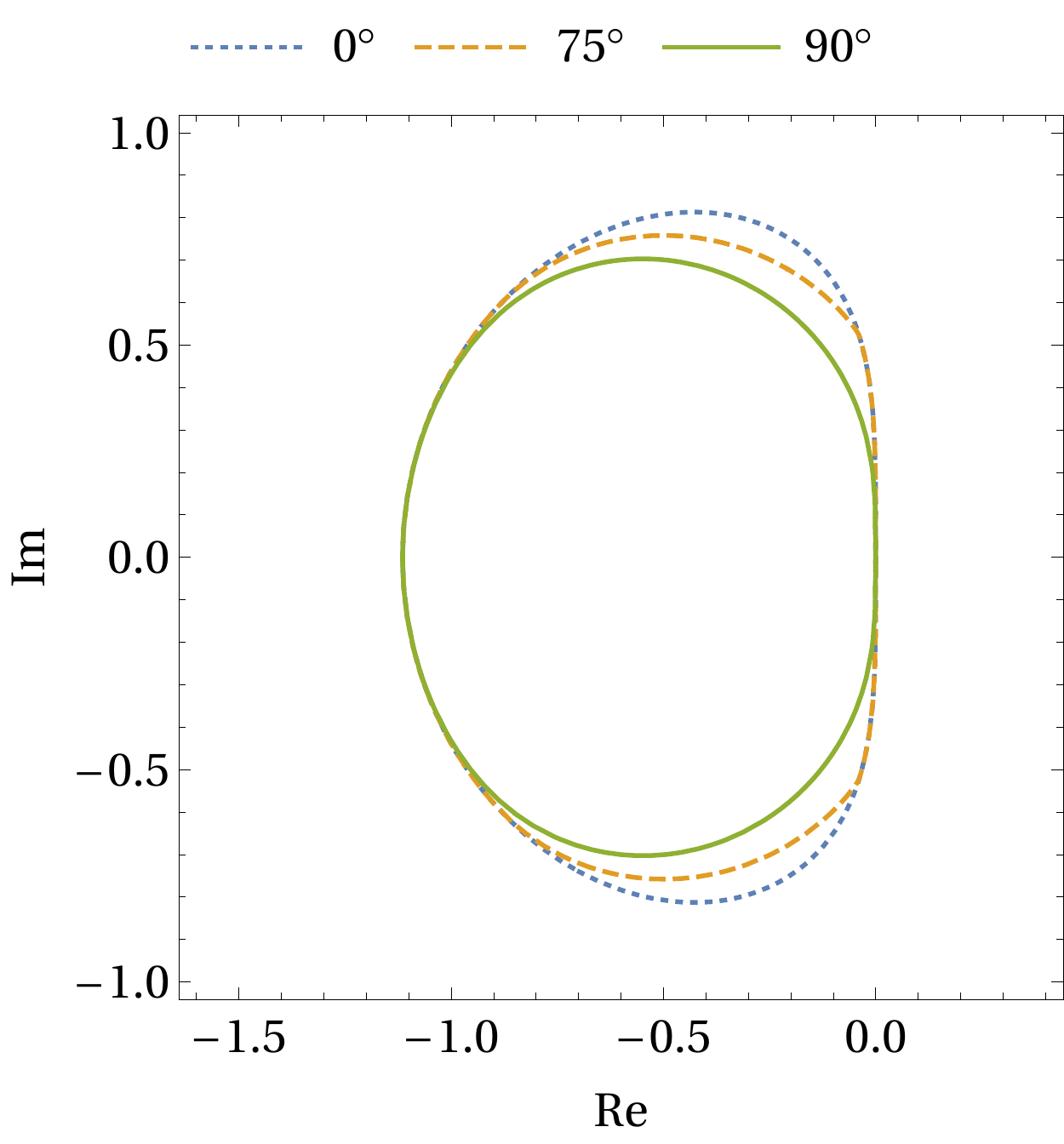}
		\caption{Second order with $\lambda \approx 0.633975$}
	\end{subfigure}
	\hfil
	\begin{subfigure}[b]{.45\linewidth}
		\includegraphics[width=\linewidth]{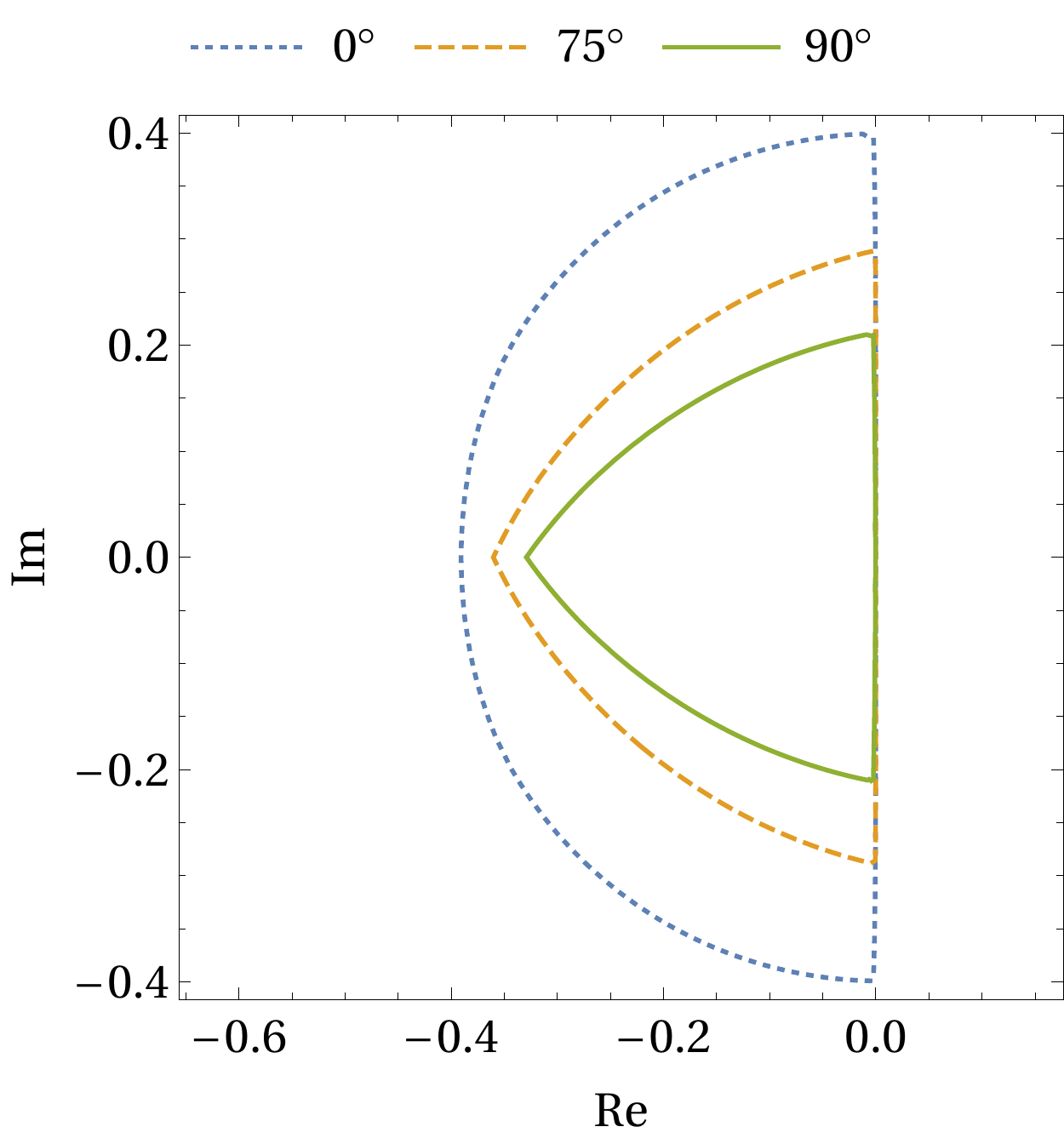}
		\caption{Fourth order with $\lambda \approx 0.872421$}
	\end{subfigure}
	\hfil
	\begin{subfigure}[b]{.45\linewidth}
		\includegraphics[width=\linewidth]{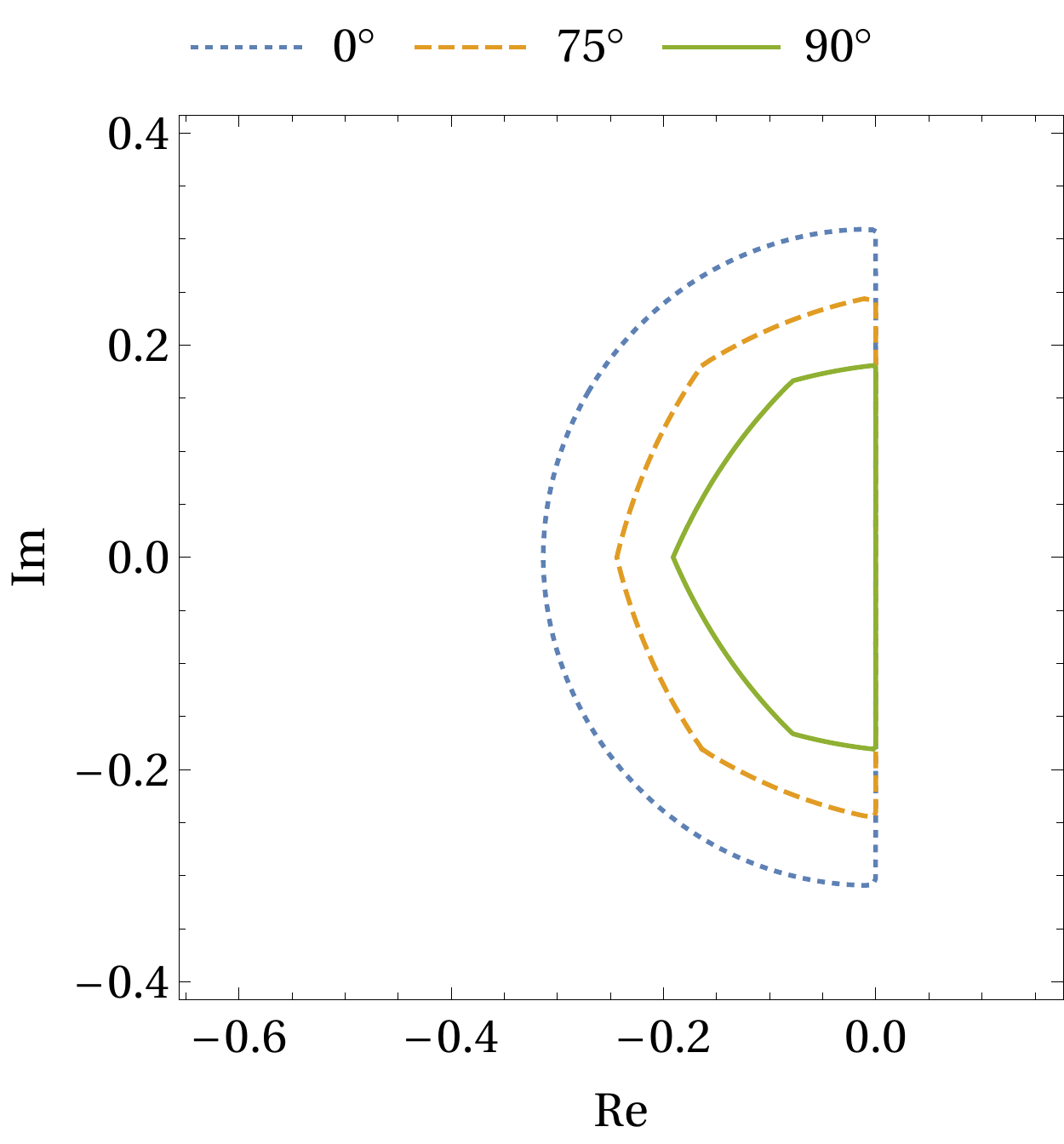}
		\caption{Seventh order with $\lambda \approx 1.35220$}
	\end{subfigure}
	\hfil
	\begin{subfigure}[b]{.45\linewidth}
		\includegraphics[width=\linewidth]{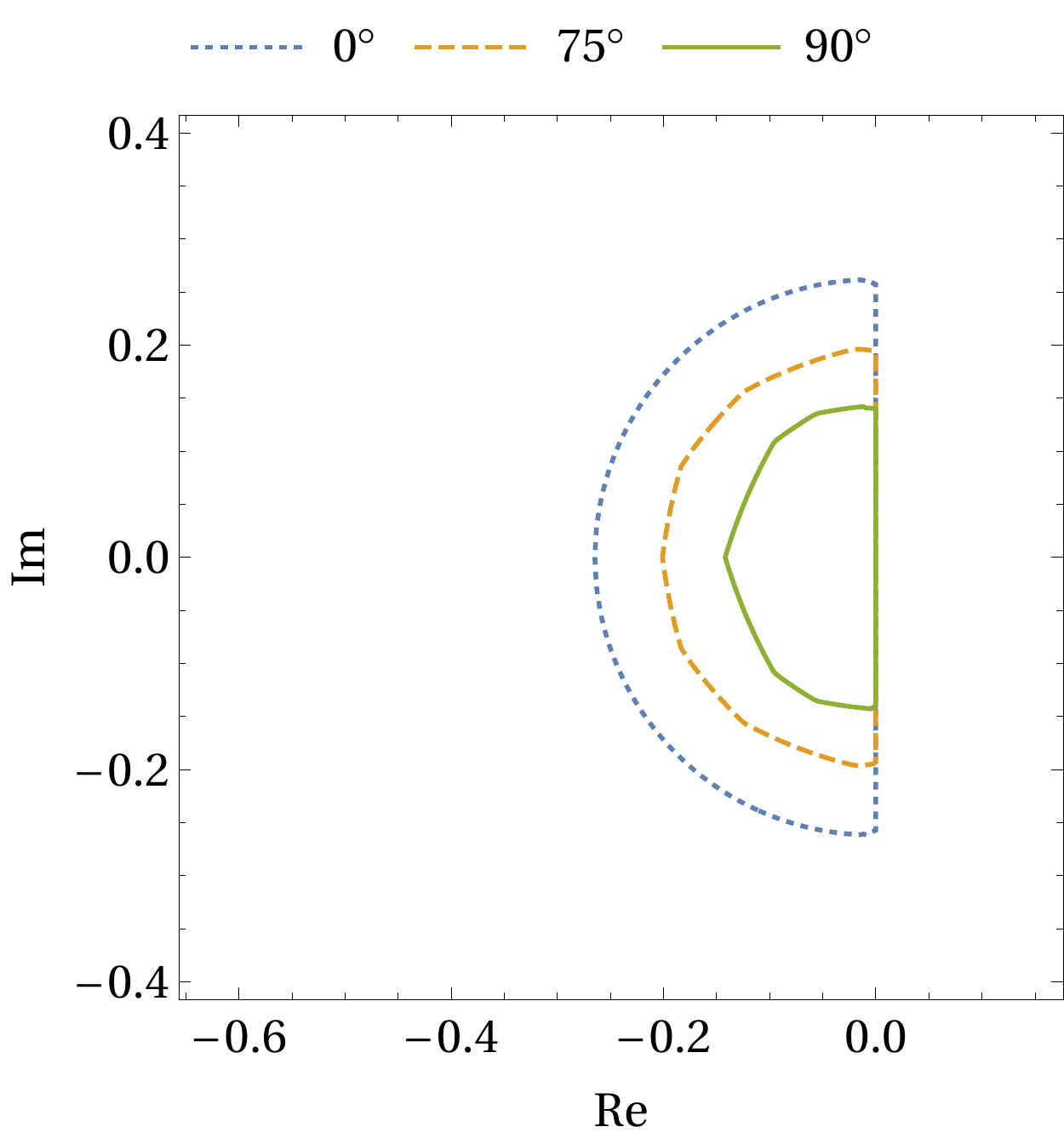}
		\caption{Tenth order with $\lambda \approx 1.69561$}
	\end{subfigure}
	\caption{Stability regions \sRegionNonstiff{} with $\alpha = \ang{0}, \ang{75}, \ang{90}$ for parallel IMEX DIMSIMs.  Note the scale for (a) is different than for the other plots.}
	\label{fig:parallel_imex_dimsim_stability}
\end{figure}

\section{Parallel ensemble IMEX Euler methods}
\label{sec:parallel_ensemble}

If one seeks to minimize communication costs for parallel IMEX GLMs, the choice $\U = \V = \eye{s}$ is attractive, as it eliminates the need to share external stages among parallel processes.  As we will show in this section, this choice of coefficients also leads to particularly favorable structures for the order conditions and stability matrix.

\begin{theorem}[Parallel ensemble IMEX Euler order conditions] \label{thm:parallel_ensemble_coeffs}
	A nonconfluent parallel ensemble IMEX Euler method, which starts with the structural assumptions
	\begin{equation*}
		\AE = \zeros{s}, \quad
		\AI = \lambda \, \eye{s}, \quad
		\U = \V = \eye{s},
	\end{equation*}
	has $p=q=r=s$ if and only if the remaining method coefficients are
	\begin{subequations} \label{eq:parallel_ensemble_coeffs}
		\begin{alignat}{2}
			\QE &= \C{s+1}, \qquad
			& \QI &= \C{s+1} - \lambda \, \C{s+1} \, \shiftmat{s+1}, \label{eq:parallel_ensemble_coeffs:q} \\
			\BE &= \C{s} \, \phishiftmat{s} \, \C{s}^{-1}, \quad
			&\BI &= \C{s} \, \phishiftmat{s} \left( \eye{s} - \lambda \, \shiftmat{s} \right) \C{s}^{-1},  \label{eq:parallel_ensemble_coeffs:b}
		\end{alignat}
	\end{subequations}
	where
	\begin{equation} \label{eq:phi_shift_mat}
		\def\arraystretch{1.2}
		\phishiftmat{n} = \begin{bmatrix}
			1 & \frac{1}{2} & \frac{1}{6} & \dots & \frac{1}{n!} \\
			& 1 & \frac{1}{2} & \dots & \frac{1}{(n-1)!} \\
			& & \ddots & \ddots & \vdots \\
			& & & 1 & \frac{1}{2} \\
			& & & & 1
		\end{bmatrix} \in \R^{n \times n}.
	\end{equation}
\end{theorem}

\begin{remark}
	An alternative representation for \cref{eq:phi_shift_mat} is $\phishiftmat{n} = \varphi_1(\shiftmat{n})$, where $\varphi_1$ is the entire function
	\begin{equation*}
		\varphi_1(z) = \sum_{k=0}^{\infty} \frac{z^k}{(k+1)!} = \frac{e^z - 1}{z}.
	\end{equation*}
\end{remark}

\begin{proof}
	With \cref{thm:parallel_imex_glm_b}, we need only show the explicit base method for parallel ensemble IMEX Euler has $p=q=r=s$ and $\BE$ and $\BI$ are related by \cref{eq:parallel_imex_glm_b:ex:b}.  By \cref{lem:parallel_imex_glm_q}, the internal stage order condition for the explicit method, given in \cref{eq:imex_glm_oc:int_e}, holds.  For the external stage order conditions:
	\begin{align*}
		\QE \, \expshiftmat{s+1} - \BE \, \C{s+1} \, \shiftmat{s+1} - \V \, \QE
		&= \C{s+1} \, \expshiftmat{s+1} - \C{s} \, \phishiftmat{s} \, \C{s}^{-1} \, \C{s+1} \, \shiftmat{s+1} - \C{s+1} \\
		&= \C{s+1} \, \expshiftmat{s+1} - \C{s+1} \, \phishiftmat{s+1} \, \shiftmat{s+1} - \C{s+1} \\
		&= \C{s+1} \left( \expshiftmat{s+1} - \phishiftmat{s+1} \, \shiftmat{s+1} - \eye{(s+1)} \right) \\
		&= \zeros{s}{(s+1)}.
	\end{align*}
	Therefore the explicit method satisfies all order conditions and has $p=q=r=s$.  Finally, 
	\begin{align*}
		\BE - \lambda \, \C{s} \, \expshiftmat{s} \, \C{s}^{-1} + \lambda \, \V
		&= \C{s} \, \phishiftmat{s} \, \C{s}^{-1} - \lambda \, \C{s} \, \expshiftmat{s} \, \C{s}^{-1} + \lambda \, \eye{s} \\
		&= \C{s} \left( \phishiftmat{s} - \lambda \, \expshiftmat{s} + \lambda \, \eye{s} \right) \C{s}^{-1} \\
		&= \C{s} \left( \phishiftmat{s} - \phishiftmat{s} \, \shiftmat{s} \right) \C{s}^{-1} \\
		&= \C{s} \, \phishiftmat{s} \left( \eye{s} - \lambda \, \shiftmat{s} \right) \C{s}^{-1} \\
		&= \BI,
	\end{align*}
	which completes the proof.
\end{proof}

While the parallel IMEX DIMSIMs of \cref{sec:parallel_imex_dimsim} require symbolic tools to derive and have coefficients that can be expressed as roots of polynomials, ensemble methods have simple, rational coefficients that can be derived with basic matrix multiplication.  The following parallel ensemble IMEX Euler method, for example, is second order:
\begin{equation*}
	\begin{butchertableau}{c|cc|cc|cc}
		0 & 0 & 0 & 1 & 0 & 1 & 0 \\
		1 & 0 & 0 & 0 & 1 & 0 & 1 \\
		\hline
		& \frac{1}{2} & \frac{1}{2} & \frac{3}{2} & -\frac{1}{2} & 1 & 0 \\
		& -\frac{1}{2} & \frac{3}{2} & \frac{1}{2} & \frac{1}{2} & 0 & 1 \\
	\end{butchertableau}.
\end{equation*}
A third order method is given by
\begin{equation*}
	\begin{butchertableau}{c|ccc|ccc|ccc}
		0 & 0 & 0 & 0 & 1 & 0 & 0 & 1 & 0 & 0 \\
		\frac{1}{2} & 0 & 0 & 0 & 0 & 1 & 0 & 0 & 1 & 0 \\
		1 & 0 & 0 & 0 & 0 & 0 & 1 & 0 & 0 & 1 \\
		\hline
		& \frac{1}{6} & \frac{2}{3} & \frac{1}{6} & \frac{7}{6} & \frac{2}{3} & -\frac{5}{6} & 1 & 0 & 0 \\
		& \frac{1}{6} & -\frac{1}{3} & \frac{7}{6} & -\frac{5}{6} & \frac{11}{3} & -\frac{11}{6} & 0 & 1 & 0 \\
		& \frac{7}{6} & -\frac{10}{3} & \frac{19}{6} & -\frac{11}{6} & \frac{14}{3} & -\frac{11}{6} & 0 & 0 & 1 \\
	\end{butchertableau},
\end{equation*}
and a fourth order method is given by
\begin{equation*}
	\begin{butchertableau}{c|cccc|cccc|cccc}
		0 & 0 & 0 & 0 & 0 & 1 & 0 & 0 & 0 & 1 & 0 & 0 & 0 \\
		\frac{1}{3} & 0 & 0 & 0 & 0 & 0 & 1 & 0 & 0 & 0 & 1 & 0 & 0 \\
		\frac{2}{3} & 0 & 0 & 0 & 0 & 0 & 0 & 1 & 0 & 0 & 0 & 1 & 0 \\
		1 & 0 & 0 & 0 & 0 & 0 & 0 & 0 & 1 & 0 & 0 & 0 & 1 \\
		\hline
		& \frac{1}{8} & \frac{3}{8} & \frac{3}{8} & \frac{1}{8} & \frac{9}{8} & \frac{3}{8} & \frac{3}{8} & -\frac{7}{8} & 1 & 0 & 0 & 0 \\
		& -\frac{1}{8} & \frac{5}{8} & -\frac{3}{8} & \frac{7}{8} & \frac{7}{8} & -\frac{19}{8} & \frac{45}{8} & -\frac{25}{8} & 0 & 1 & 0 & 0 \\
		& -\frac{7}{8} & \frac{27}{8} & -\frac{37}{8} & \frac{25}{8} & \frac{25}{8} & -\frac{93}{8} & \frac{131}{8} & -\frac{55}{8} & 0 & 0 & 1 & 0 \\
		& -\frac{25}{8} & \frac{93}{8} & -\frac{123}{8} & \frac{63}{8} & \frac{55}{8} & -\frac{195}{8} & \frac{237}{8} & -\frac{89}{8} & 0 & 0 & 0 & 1 \\
	\end{butchertableau}.
\end{equation*}

When the order of the method increases, so does the magnitude of the method coefficients: a phenomenon previously described for parallel IMEX DIMSIMs.  Similarly, the distribution of abscissae can limit the growth of coefficients, and thus, the floating-point errors associated with them.  \Cref{tab:parallel_ensemble_max_coeff} lists these maximum coefficients for $\c$'s evenly space between $[0,1]$, as well as $[2-s,1]$.

\begin{table}
	\centering
	\begin{tabular}{c|c|c}
		Method order & $c_i = \frac{i-1}{s-1}$ & $c_i = 1 - s + i$ \\ \hline
		2 & 1.50 & 1.50 \\
		3 & 4.67 & 1.92 \\
		4 & 29.62 & 3.54 \\
		5 & 203.87 & 6.37 \\
		6 & 1380.73 & 13.07 \\
		7 & 9868.32 & 23.62 \\
		8 & 69256.88 & 47.97 \\
		9 & 506662.23 & 87.98 \\
		10 & 3639853.98 & 177.82
	\end{tabular}
	\caption{Approximate values for the largest coefficient in absolute value from $\BE$ and $\BI$ for parallel ensemble IMEX Euler methods of orders two to ten with $\lambda=1$.}
	\label{tab:parallel_ensemble_max_coeff}
\end{table}

\subsection{Stability}

An interesting property of parallel ensemble IMEX Euler methods is that $\BE$, $\BI$, $\AE$, and $\AI$ all simultaneously triangularize.  The stability matrix \cref{eq:parallel_imex_glm_stability} can therefore be put into an upper triangular form with a simple similarity transformation:
\begin{equation} \label{eq:transformed_stability}
	\C{s}^{-1} \, \M(w, \widehat{w}) \, \C{s} = \eye{r} + \frac{w}{1 - \lambda \, \widehat{w}} \, \phishiftmat{s} + \frac{\widehat{w}}{1 - \lambda \, \widehat{w}} \, \phishiftmat{s} \left( \eye{s} - \lambda \, \shiftmat{s} \right).
\end{equation}
The diagonal entries of \cref{eq:transformed_stability} are all  $1 + (w+\widehat{w})/(1 - \lambda \, \widehat{w}) $ and identically are the eigenvalues of the stability matrix.  Note the geometric multiplicity of this repeated eigenvalue is $r$ when $w = \widehat{w} = 0$ and $1$ otherwise. In order to ensure L-stability of the implicit base method as well as $\rho(\M(w, \infty)) = 0$, we set $\lambda = 1$.  In this case, the eigenvalues simplify to $(1+w)/(1-\widehat{w})$ matching the stability of the IMEX Euler scheme
\begin{equation*}
	y_{n} = y_{n-1} + h \, f(t_{n-1}, y_{n-1}) + h\, g(t_{n}, y_{n}).
\end{equation*}

There are several other interesting stability features for parallel ensemble IMEX Euler methods.  First, stability is independent of the order and the choice of abscissae, allowing a systematic approach to develop stable methods of arbitrary order.  The constrained nonstiff stability regions has the simple form
\begin{equation*}
	\sRegionNonstiff = \sRegionE = \left\{ w \in \Cplx : \abs{1 + w} < 1 \vee w=0 \right\},
\end{equation*}
when $s > 1$.  Except for the origin, the boundary of this circular stability region is carefully excluded because the $1$ eigenvalue of $\M$ is defective at those points.  This family of methods is stability decoupled in the sense that linear stability of the base methods for their respective partitions implies linear stability of the IMEX scheme.

We note that aside from the origin, \sRegionNonstiff{} does not contain any of the imaginary axis, indicating potential stability issues when $f$ is oscillatory.  This analysis is a bit pessimistic, however, as \sRegionNonstiff{} represents the explicit stability when $\widehat{w}$ is chosen in a worst-case scenario.  Only when $\widehat{w}=0$ is there instability for all purely imaginary $w$.  As the modulus of $\widehat{w}$ grows, the range of imaginary $w$ for which the IMEX method is stable also grows.
\section{Numerical experiments}
\label{sec:experiments}

We provide numerical experiments to confirm the order of convergence  and to study the performance of our methods compared to other IMEX methods. We use the CUSP and Allen--Cahn problems in our experiments.

\subsection{CUSP problem}
\label{subsec:cusp}

The CUSP problem \cite[Chapter IV.10]{Hairer_book_II} is associated with the equations
\begin{align} \label{eq:cusp}
	\begin{split}
	 	\pdv{y}{t} &= -\frac{1}{\varepsilon} \left( y^3 + a \, y + b \right) + \sigma \, \pdv[2]{y}{x}, \\
	 	\pdv{a}{t} &=  b+ 0.07 \, v +  \sigma \, \pdv[2]{a}{x},\\
	 	\pdv{b}{t} &= b \, (1-a^2) - a- 0.4 \, y + 0.035 \, v + \sigma \, \pdv[2]{b}{x},
	\end{split}
\end{align}
where $v = \frac{u}{u+ 0.1}$ and $u = (y-0.7) \, (y-1.3)$.  The timespan is $t \in [0,1.1]$, the spatial domain is $x \in [0,1]$, and the parameters are chosen as $\sigma = \frac{1}{144}$ and $\varepsilon = 10^{-4}$. 
Spatial derivatives are discretized using second order central finite differences on a uniform mesh with $N=32$ points and periodic boundary conditions. The initial conditions are
\begin{equation*}
	y_i(0) =0, \qquad a_i(0) = -2 \cos(\frac{2\pi i}{N}), \qquad b_i(0) = 2\sin(\frac{2\pi i}{N}), 
\end{equation*}
for $i=1,\dots,N$.  Note that the problem is singularly perturbed in the $y$ component and the stiffness of the system can be controlled using $\varepsilon$.  Following the splitting used in \cite{JACKIEWICZ2017}, the diffusion terms and the term scaled by $\varepsilon^{-1}$ form $g$, while the remaining terms form $f$.  The MATLAB implementation of the CUSP problem is available in \cite{otp,otpsoft}.

We performed a fixed time-stepping convergence study of the new methods. \Cref{fig:CUSP-convergence} shows the error of the final solution versus number of timesteps.  Error is computed in the $\ell^2$ sense using a high-accuracy reference solution.  In all cases, the parallel IMEX GLMs converge at least at the same rate as the theoretical order of accuracy.

\begin{figure}
	\centering
	\begin{subfigure}[t]{.48\linewidth}
		\begin{tikzpicture}
			\begin{loglogaxis}[xlabel={Steps},ylabel={Error},
			legend entries={Parallel IMEX DIMSIM2, Parallel IMEX DIMSIM7}]
			\addplot table[x index=0,y index=1, col sep=comma] {./Data/Convergence/IMEX-DIMSIM2.txt};
			\addplot table[x index=0,y index=1, col sep=comma] {./Data/Convergence/IMEX-DIMSIM7.txt};
			\draw[dashed] (axis cs:32e4, 3e-7) -- node[above]{\scriptsize $2$} (axis cs:48e4, 1.333e-7);
			\draw[dashed] (axis cs:2e5,32e-6) -- node[above]{\scriptsize $7$} (axis cs:2.8e5, 3.0357e-06);
			\end{loglogaxis}
		\end{tikzpicture}
		\caption{Error versus steps for parallel IMEX DIMSIMs of orders two and seven.  For this problem, the seventh order method converges faster than the nominal order.}
	\end{subfigure}
	\hfil
	\begin{subfigure}[t]{.48\linewidth}
		\begin{tikzpicture}
			\begin{loglogaxis}[xlabel={Steps},ylabel={Error},
			legend entries={Parallel Ensemble IMEX Euler3, Parallel Ensemble IMEX Euler8}]
			\addplot[mark=triangle,color=teal] table[x index=0,y index=1, col sep=comma] {./Data/Convergence/Ensemble3.txt};
			\addplot[mark=asterisk,color=black] table[x index=0,y index=1, col sep=comma] {./Data/Convergence/Ensemble8.txt};
			\draw[dashed] (axis cs:1e6,27e-10) -- node[above]{\scriptsize $3$} (axis  cs:3e6, 1e-10);
			\draw[dashed] (axis cs:3e5,2e-10) -- node[below]{\scriptsize $8$} (axis cs:6e5, 1.5625e-12);
			\end{loglogaxis}
		\end{tikzpicture}
		\caption{Error versus steps for parallel ensemble IMEX Euler methods of orders three and eight.}
	\end{subfigure}
\caption{Convergence of parallel IMEX DIMSIM and parallel ensemble IMEX Euler methods for the CUSP problem \cref{eq:cusp}.}
\label{fig:CUSP-convergence}
\end{figure}
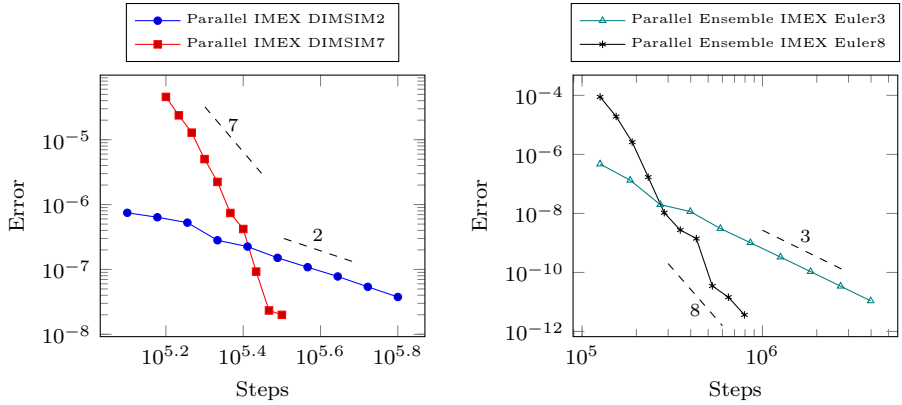

\subsection{Allen--Cahn problem}
\label{subsec:AllenCahn}

We also consider the two-dimensional Allen--Cahn problem described in \cite{Sandu_2016_highOrderIMEX-GLM}.  It is a reaction-diffusion system governed by the equation
\begin{equation} \label{eq:allen-cahn}
	\pdv{u}{t} = \alpha \, \laplacian{u} + \beta \, (u - u^3) + s,
\end{equation}
where $\alpha = 0.1$ and $\beta = 3$.  The time-dependent Dirichlet boundary conditions and source term $s(t,x,y)$ are derived using method of manufactured solutions such that the exact solution is
\begin{equation*}
	u(t,x,y) = 2 + \sin(2\pi \left( x -t \right)) \cos(3 \pi \left(y-t \right)).
\end{equation*}
We discretize the PDE on a unit square domain using degree two Lagrange finite elements and a uniform triangular mesh with $N=32$ points in each direction.  The diffusion term and forcing associated with the boundary conditions are treated implicitly, while the reaction and source term are treated explicitly.

The problem is implemented using the FEniCS package \cite{Fenics2015a} leveraging OpenMP parallelism to speed up $f$ and $g$ evaluations, as well as MPI parallelism of stage computations made possible by the structure of the parallel IMEX GLMs.  All tests were run on the Cascades cluster maintained by Virginia Tech's Advance Research Computing center (ARC). Parallel experiments were performed on $p=q=r=s$ nodes, each using 12 cores. Serial experiments were done on a single node with the same number of cores. The error was computed using the $\ell_2$ norm by comparing the nodal values of the numerical solution against a high-accuracy reference solution.

\Cref{fig:Allen-Cahn-performance} summarizes the results of this experiment by comparing several additive Runge--Kutta (ARK) methods and IMEX DIMSIMs with Parallel IMEX GLMs derived in this paper. At order three, serial methods are ARK3(2)4L[2]SA from \cite{Kennedy2003} and IMEX-DIMSIM3 from \cite{Sandu_2015_Stable_IMEX-GLM}. At order four, comparisons are done against  ARK4(3)7L[2]SA\textsubscript{1} from \cite{KENNEDY2019} and IMEX-DIMSIM4 from \cite{Sandu_2016_highOrderIMEX-GLM}. Order five serial methods are ARK5(4)8L[2]SA\textsubscript{2} from \cite{KENNEDY2019} and IMEX-DIMSIM5 from \cite{Sandu_2016_highOrderIMEX-GLM}. Finally, the order six baseline is IMEX-DIMSIM6($\sRegionNonstiff[\pi/2]$) from \cite{JACKIEWICZ2017}.  The results show parallel ensemble IMEX Euler methods are the most efficient in all cases. Parallel IMEX DIMSIMs are competitive at orders three and six and surpass the efficiency of serial schemes at orders four and five.

\Cref{fig:Allen-Cahn-convergence} plots convergence of the methods used in the experiment. We can see the ARK methods exhibit order reduction for this problem, which explains their poor efficiency results.  All other methods achieve the expected order of accuracy.  For a fixed number of steps, the parallel IMEX GLMs are less accurate than the serial IMEX GLMs, which indicates parallel methods have larger error constants and are not the most efficient when limited to serial execution.  This is to be expected given parallel methods have a more restrictive structure and less coefficients available for optimizing the principal error and stability.

\begin{figure}
	\centering
	\begin{subfigure}[b]{.49\linewidth}
		\begin{tikzpicture}
		\begin{loglogaxis}[xlabel={Runtime (s)},ylabel={Error},
		legend entries={ Parallel Ensemble IMEX Euler3, Parallel IMEX DIMSIM3,IMEX-DIMSIM3, ARK3(2)4L[2]SA }]
		\addplot table[x index=2,y index=1] {./Data/Performance/Order3/Parallel_Ensemble_IMEX_Euler3.dat};
		\addplot table[x index=2,y index=1] {./Data/Performance/Order3/Parallel_IMEX-DIMSIM3.dat};
		\addplot table[x index=2,y index=1] {./Data/Performance/Order3/IMEX-DIMSIM3.dat};
		\addplot table[x index=2,y index=1] {./Data/Performance/Order3/ARK324L2SA.dat};
		\end{loglogaxis}
		\end{tikzpicture}
		\caption{ Order 3}
	\end{subfigure}
	\hfill
	\begin{subfigure}[b]{.49\linewidth}
		\begin{tikzpicture}
		\begin{loglogaxis}[xlabel={Runtime (s)},ylabel={Error},
		legend entries={Parallel Ensemble IMEX Euler4, Parallel IMEX DIMSIM4,IMEX-DIMSIM4,ARK4(3)7L[2]SA\textsubscript{1} }]
		\addplot table[x index=2,y index=1] {./Data/Performance/Order4/Parallel_Ensemble_IMEX_Euler4.dat};
		\addplot table[x index=2,y index=1] {./Data/Performance/Order4/Parallel_IMEX-DIMSIM4.dat};
		\addplot table[x index=2,y index=1] {./Data/Performance/Order4/IMEX-DIMSIM4.dat};
		\addplot table[x index=2,y index=1] {./Data/Performance/Order4/ARK437L2SA_1.dat};
		\end{loglogaxis}
		\end{tikzpicture}
		\caption{ Order 4}
	\end{subfigure}
	
	\begin{subfigure}[b]{.49\linewidth}
		\begin{tikzpicture}
		\begin{loglogaxis}[xlabel={Runtime (s)},ylabel={Error},
		legend entries={Parallel Ensemble IMEX Euler5,Parallel IMEX DIMSIM5,  IMEX-DIMSIM5,ARK5(4)8L[2]SA\textsubscript{2}}]
		\addplot table[x index=2,y index=1] {./Data/Performance/Order5/Parallel_Ensemble_IMEX_Euler5.dat};
		\addplot table[x index=2,y index=1] {./Data/Performance/Order5/Parallel_IMEX-DIMSIM5b.dat};
		\addplot table[x index=2,y index=1] {./Data/Performance/Order5/IMEX-DIMSIM5.dat};
		\addplot table[x index=2,y index=1] {./Data/Performance/Order5/ARK548L2SA_2.dat};
		\end{loglogaxis}
		\end{tikzpicture}
		\caption{ Order 5}
	\end{subfigure}
	\hfill
	\begin{subfigure}[b]{.49\linewidth}
		\begin{tikzpicture}
		\begin{loglogaxis}[xlabel={Runtime (s)},ylabel={Error},
		legend entries={Parallel Ensemble IMEX Euler6,Parallel IMEX DIMSIM6, IMEX-DIMSIM6($\sRegionNonstiff[\pi/2]$)}]
		\addplot table[x index=2,y index=1] {./Data/Performance/Order6/Parallel_Ensemble_IMEX_Euler6.dat};
		\addplot table[x index=2,y index=1] {./Data/Performance/Order6/Parallel_IMEX-DIMSIM6b.dat};
		\addplot table[x index=2,y index=1] {./Data/Performance/Order6/JACKW-DIMSIM6.dat};
		\end{loglogaxis}
		\end{tikzpicture}
		\caption{ Order 6}
	\end{subfigure}
	\caption{Work-precision diagrams for the Allen--Cahn problem \cref{eq:allen-cahn}}
	\label{fig:Allen-Cahn-performance}
\end{figure}
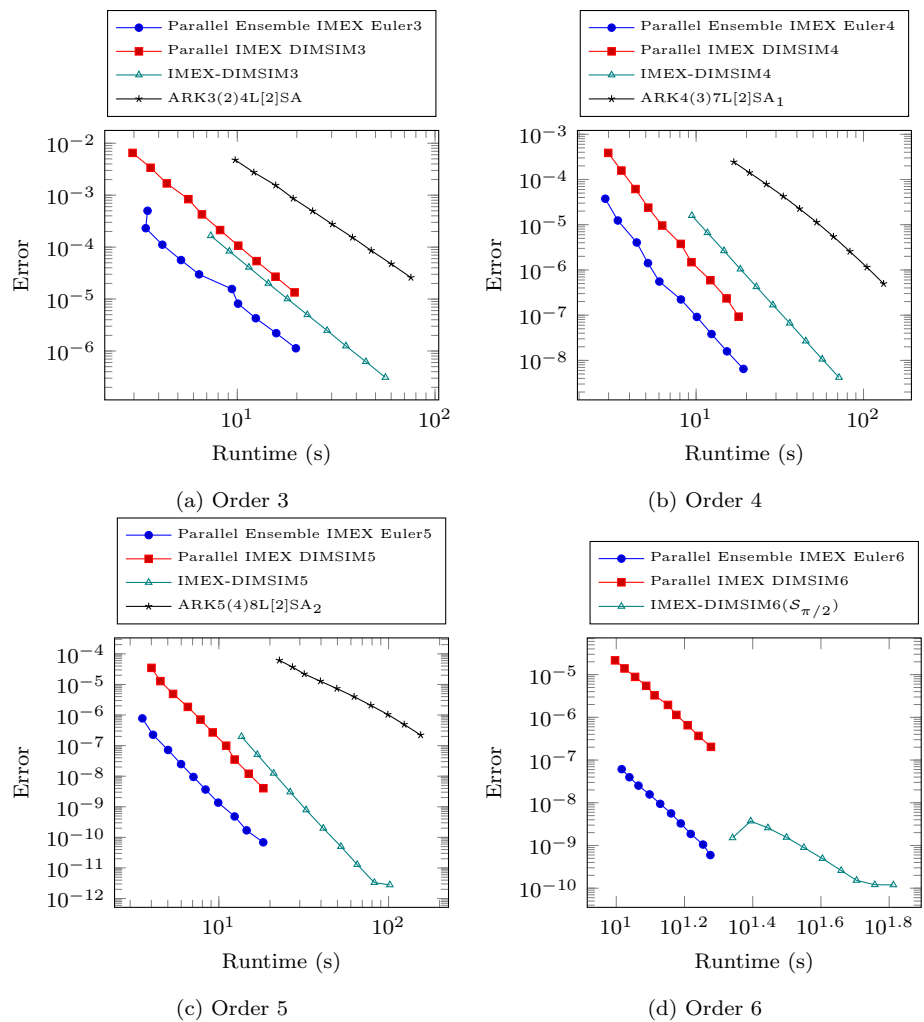
%
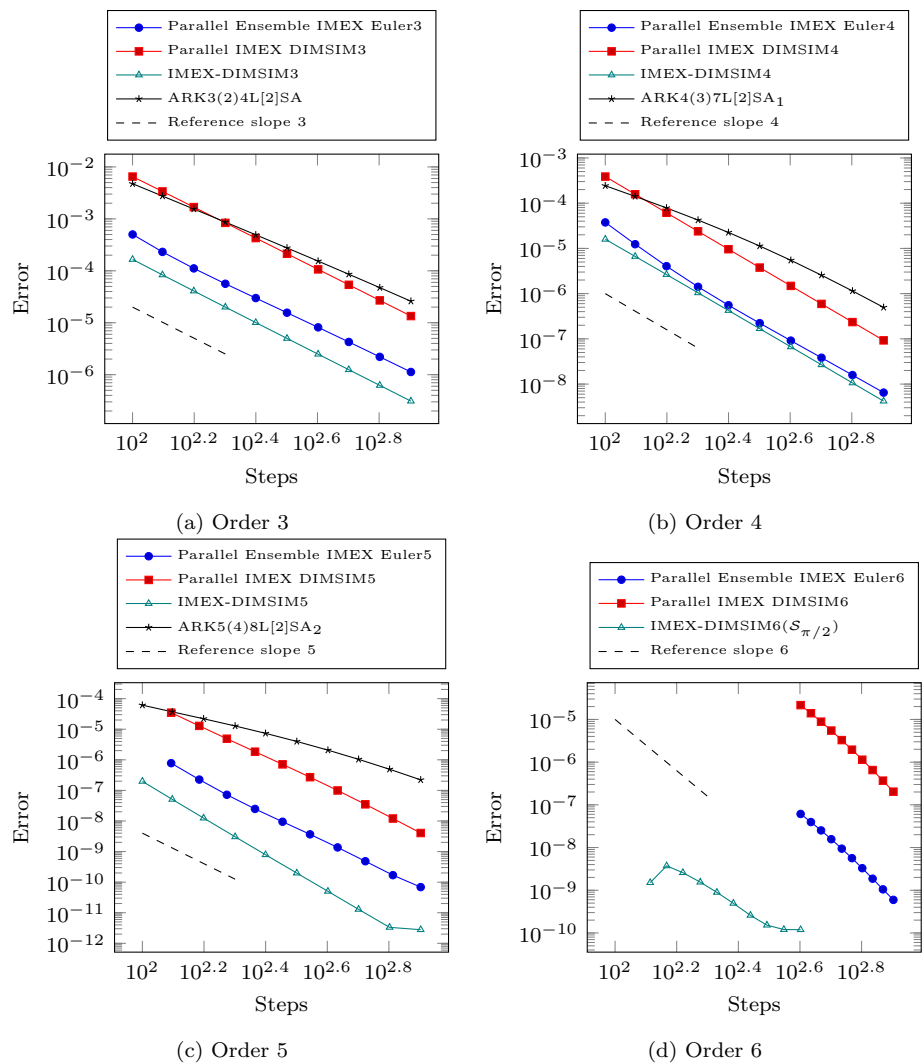
\begin{figure}
	\centering
	\begin{subfigure}[b]{.49\linewidth}
		\begin{tikzpicture}
		\begin{loglogaxis}[xlabel={Steps},ylabel={Error},
		legend entries={ Parallel Ensemble IMEX Euler3, Parallel IMEX DIMSIM3,IMEX-DIMSIM3, ARK3(2)4L[2]SA , Reference slope 3}]
		\addplot table[x index=0,y index=1] {./Data/Performance/Order3/Parallel_Ensemble_IMEX_Euler3.dat};
		\addplot table[x index=0,y index=1] {./Data/Performance/Order3/Parallel_IMEX-DIMSIM3.dat};
		\addplot table[x index=0,y index=1] {./Data/Performance/Order3/IMEX-DIMSIM3.dat};
		\addplot table[x index=0,y index=1] {./Data/Performance/Order3/ARK324L2SA.dat};
		\addplot[dashed] coordinates {(1e2, 20e-6) (2e2,20/8*1e-6) };
		\end{loglogaxis}
		\end{tikzpicture}
		\caption{ Order 3}
	\end{subfigure}
	\hfill
	\begin{subfigure}[b]{.49\linewidth}
		\begin{tikzpicture}
		\begin{loglogaxis}[xlabel={Steps},ylabel={Error},
		legend entries={Parallel Ensemble IMEX Euler4, Parallel IMEX DIMSIM4,IMEX-DIMSIM4,ARK4(3)7L[2]SA\textsubscript{1}, Reference slope 4}]
		\addplot table[x index=0,y index=1] {./Data/Performance/Order4/Parallel_Ensemble_IMEX_Euler4.dat};
		\addplot table[x index=0,y index=1] {./Data/Performance/Order4/Parallel_IMEX-DIMSIM4.dat};
		\addplot table[x index=0,y index=1] {./Data/Performance/Order4/IMEX-DIMSIM4.dat};
		\addplot table[x index=0,y index=1] {./Data/Performance/Order4/ARK437L2SA_1.dat};
		\addplot[dashed] coordinates {(1e2, 1e-6) (2e2,1/16*1e-6) };
		\end{loglogaxis}
		\end{tikzpicture}
		\caption{ Order 4}
	\end{subfigure}
	
	\begin{subfigure}[b]{.49\linewidth}
		\begin{tikzpicture}
		\begin{loglogaxis}[xlabel={Steps},ylabel={Error},
		legend entries={Parallel Ensemble IMEX Euler5,Parallel IMEX DIMSIM5,  IMEX-DIMSIM5,ARK5(4)8L[2]SA\textsubscript{2}, Reference slope 5}]
		\addplot table[x index=0,y index=1] {./Data/Performance/Order5/Parallel_Ensemble_IMEX_Euler5.dat};
		\addplot table[x index=0,y index=1] {./Data/Performance/Order5/Parallel_IMEX-DIMSIM5b.dat};
		\addplot table[x index=0,y index=1] {./Data/Performance/Order5/IMEX-DIMSIM5.dat};
		\addplot table[x index=0,y index=1] {./Data/Performance/Order5/ARK548L2SA_2.dat};
		\addplot[dashed] coordinates {(1e2, 4e-9) (2e2,4/32*1e-9) };
		\end{loglogaxis}
		\end{tikzpicture}
		\caption{ Order 5}
	\end{subfigure}
	\hfill
	\begin{subfigure}[b]{.49\linewidth}
		\begin{tikzpicture}
		\begin{loglogaxis}[xlabel={Steps},ylabel={Error},
		legend entries={Parallel Ensemble IMEX Euler6,Parallel IMEX DIMSIM6, IMEX-DIMSIM6($\sRegionNonstiff[\pi/2]$), Reference slope 6}]
		\addplot table[x index=0,y index=1] {./Data/Performance/Order6/Parallel_Ensemble_IMEX_Euler6.dat};
		\addplot table[x index=0,y index=1] {./Data/Performance/Order6/Parallel_IMEX-DIMSIM6b.dat};
		\addplot table[x index=0,y index=1] {./Data/Performance/Order6/JACKW-DIMSIM6.dat};
		\addplot[dashed] coordinates {(1e2, 1e-5) (2e2,1/64*1e-5) };
		\end{loglogaxis}
		\end{tikzpicture}
		\caption{ Order 6}
	\end{subfigure}
	\caption{Convergence diagrams for the Allen--Cahn problem \cref{eq:allen-cahn}}
	\label{fig:Allen-Cahn-convergence}
\end{figure}
%
%

\section{Conclusion}
\label{sec:conclusion}

This paper studies parallel IMEX GLMs and provides a methodology to derive and solve simple order conditions for methods of arbitrary order. Using this framework, we construct two families of methods, based on existing DIMSIMs and on IMEX Euler, and provide linear stability analyses for them. 

Our numerical experiments show that parallel IMEX GLMs can outperform existing serial IMEX schemes.  Between parallel IMEX DIMSIMs and parallel ensemble IMEX Euler methods, the latter proved to be the most competitive.  The error for the ensemble methods is generally smaller than that of the DIMSIMs, due in part to the improved ending procedure.  Moreover, the magnitude of method coefficients grows slower for ensemble methods as documented in \cref{tab:imex_dimsim_max_coeff,tab:parallel_ensemble_max_coeff}, reducing the impact of accumulated floating-point errors.  For orders five and higher we have to carefully select the method and distribution of the abscissae to control these errors.  In addition, one notes that parallel ensemble IMEX Euler methods tend to have smaller values of $\lambda$, which improves convergence of iterative linear solvers used in the Newton iterations.

Owing to their excellent stability properties, the ensemble family shows great potential for constructing other types of partitioned GLMs.  Of particular interest are alternating directions implicit (ADI) GLMs \cite{Sarshar2019}, as well as multirate GLMs. The authors hope to study these in future works.

\begin{acknowledgements}
	The authors acknowledge Advanced Research Computing at Virginia Tech for providing computational resources and technical support that have contributed to the results reported within this paper. URL: \url{http://www.arc.vt.edu}
	
	On behalf of all authors, the corresponding author states that there is no conflict of interest.
\end{acknowledgements}

\bibliographystyle{spmpsci}
\bibliography{Bib/glm,Bib/imex,Bib/misc,Bib/ode_general,Bib/sandu,Bib/imex_glm,Bib/ark}

\begin{thebibliography}{10}
\providecommand{\url}[1]{{#1}}
\providecommand{\urlprefix}{URL }
\expandafter\ifx\csname urlstyle\endcsname\relax
  \providecommand{\doi}[1]{DOI~\discretionary{}{}{}#1}\else
  \providecommand{\doi}{DOI~\discretionary{}{}{}\begingroup
  \urlstyle{rm}\Url}\fi

\bibitem{Fenics2015a}
Aln{\ae}s, M.S., Blechta, J., Hake, J., Johansson, A., Kehlet, B., Logg, A.,
  Richardson, C., Ring, J., Rognes, M.E., Wells, G.N.: The {FEniCS} project
  version 1.5.
\newblock Archive of Numerical Software \textbf{3}(100) (2015).
\newblock \doi{10.11588/ans.2015.100.20553}

\bibitem{ascher1997implicit}
Ascher, U.M., Ruuth, S.J., Spiteri, R.J.: Implicit-explicit {Runge--Kutta}
  methods for time-dependent partial differential equations.
\newblock Applied Numerical Mathematics \textbf{25}(2-3), 151--167 (1997)

\bibitem{ascher1995implicit}
Ascher, U.M., Ruuth, S.J., Wetton, B.T.: Implicit-explicit methods for
  time-dependent partial differential equations.
\newblock SIAM Journal on Numerical Analysis \textbf{32}(3), 797--823 (1995)

\bibitem{boscarino2009class}
Boscarino, S., Russo, G.: On a class of uniformly accurate {IMEX}
  {Runge--Kutta} schemes and applications to hyperbolic systems with
  relaxation.
\newblock SIAM Journal on Scientific Computing \textbf{31}(3), 1926--1945
  (2009)

\bibitem{BRAS2018207}
Bra{\'s}, M., Cardone, A., Jackiewicz, Z., Pierzcha{\l}a, P.: Error propagation
  for implicit--explicit general linear methods.
\newblock Applied Numerical Mathematics \textbf{131}, 207 -- 231 (2018).
\newblock \doi{https://doi.org/10.1016/j.apnum.2018.05.004}

\bibitem{bras2017accurate}
Bra{\'s}, M., Izzo, G., Jackiewicz, Z.: Accurate implicit--explicit general
  linear methods with inherent {Runge--Kutta} stability.
\newblock Journal of Scientific Computing \textbf{70}(3), 1105--1143 (2017)

\bibitem{butcher1993diagonally}
Butcher, J.C.: Diagonally-implicit multi-stage integration methods.
\newblock Applied Numerical Mathematics \textbf{11}(5), 347--363 (1993)

\bibitem{butcher1993general}
Butcher, J.C.: General linear methods for the parallel solution of ordinary
  differential equations.
\newblock In: Contributions In Numerical Mathematics, pp. 99--111. World
  Scientific (1993)

\bibitem{butcher1997order}
Butcher, J.C.: Order and stability of parallel methods for stiff problems.
\newblock Advances in Computational Mathematics \textbf{7}(1-2), 79--96 (1997)

\bibitem{butcher1995parallel}
Butcher, J.C., Chartier, P.: Parallel general linear methods for stiff ordinary
  differential and differential algebraic equations.
\newblock Applied Numerical Mathematics \textbf{17}(3), 213 -- 222 (1995).
\newblock \doi{10.1016/0168-9274(95)00029-T}.
\newblock Special Issue on Numerical Methods for Ordinary Differential
  Equations

\bibitem{Califano2017}
Califano, G., Izzo, G., Jackiewicz, Z.: {Starting procedures for general linear
  methods}.
\newblock Applied Numerical Mathematics \textbf{120}, 165--175 (2017).
\newblock \doi{10.1016/J.APNUM.2017.05.009}

\bibitem{Sandu_2014_IMEX-RK}
Cardone, A., Jackiewicz, Z., Sandu, A., Zhang, H.: Extrapolated {IMEX
  Runge--Kutta} methods.
\newblock Mathematical Modelling and Analysis \textbf{19}(2), 18--43 (2014).
\newblock \doi{10.3846/13926292.2014.892903}

\bibitem{Sandu_2014_IMEX_GLM_Extrap}
Cardone, A., Jackiewicz, Z., Sandu, A., Zhang, H.: Extrapolation-based
  implicit-explicit general linear methods.
\newblock Numerical Algorithms \textbf{65}(3), 377--399 (2014).
\newblock \doi{10.1007/s11075-013-9759-y}

\bibitem{Sandu_2015_Stable_IMEX-GLM}
Cardone, A., Jackiewicz, Z., Sandu, A., Zhang, H.: Construction of highly
  stable implicit-explicit general linear methods.
\newblock In: AIMS proceedings, vol. Dynamical Systems, Differential Equations,
  and Applications. Madrid, Spain (2015).
\newblock \doi{10.3934/proc.2015.0185}

\bibitem{otpsoft}
{Computational Science Laboratory}: {ODE} test problems (2020).
\newblock
  \urlprefix\url{https://github.com/ComputationalScienceLaboratory/ODE-Test-Problems}

\bibitem{Connors2011}
Connors, J.M., Miloua, A.: Partitioned time discretization for parallel
  solution of coupled {ODE} systems.
\newblock BIT Numerical Mathematics \textbf{51}(2), 253--273 (2011).
\newblock \doi{10.1007/s10543-010-0295-z}

\bibitem{Sandu_2010_extrapolatedIMEX}
Constantinescu, E., Sandu, A.: {Extrapolated implicit-explicit time stepping}.
\newblock SIAM Journal on Scientific Computing \textbf{31}(6), 4452--4477
  (2010).
\newblock \doi{10.1137/080732833}

\bibitem{ditkowski2019imex}
Ditkowski, A., Gottlieb, S., Grant, Z.J.: Imex error inhibiting schemes with
  post-processing.
\newblock arXiv preprint arXiv:1912.10027  (2019)

\bibitem{frank1997stability}
Frank, J., Hundsdorfer, W., Verwer, J.: On the stability of implicit-explicit
  linear multistep methods.
\newblock Applied Numerical Mathematics \textbf{25}(2-3), 193--205 (1997)

\bibitem{Hairer_book_II}
Hairer, E., Wanner, G.: Solving ordinary differential equations {II}: {S}tiff
  and differential-algebraic problems, 2 edn.
\newblock No.~14 in Springer Series in Computational Mathematics.
  Springer-Verlag Berlin Heidelberg (1996)

\bibitem{hundsdorfer2007imex}
Hundsdorfer, W., Ruuth, S.J.: {IMEX} extensions of linear multistep methods
  with general monotonicity and boundedness properties.
\newblock Journal of Computational Physics \textbf{225}(2), 2016--2042 (2007)

\bibitem{izzo2019transformed}
Izzo, G., Jackiewicz, Z.: Transformed implicit-explicit {DIMSIMs} with strong
  stability preserving explicit part.
\newblock Numerical Algorithms \textbf{81}(4), 1343--1359 (2019)

\bibitem{jackiewicz2009general}
Jackiewicz, Z.: General linear methods for ordinary differential equations.
\newblock John Wiley \& Sons (2009)

\bibitem{JACKIEWICZ2017}
Jackiewicz, Z., Mittelmann, H.: Construction of {IMEX DIMSIMs} of high order
  and stage order.
\newblock Applied Numerical Mathematics \textbf{121}, 234 -- 248 (2017).
\newblock \doi{10.1016/j.apnum.2017.07.004}

\bibitem{Kennedy2003}
Kennedy, C.A., Carpenter, M.H.: Additive {Runge--Kutta} schemes for
  convection-diffusion-reaction equations.
\newblock Applied Numerical Mathematics \textbf{44}(1-2), 139--181 (2003).
\newblock \doi{10.1016/S0168-9274(02)00138-1}

\bibitem{KENNEDY2019}
Kennedy, C.A., Carpenter, M.H.: Higher-order additive {Runge--Kutta} schemes
  for ordinary differential equations.
\newblock Applied Numerical Mathematics \textbf{136}, 183 -- 205 (2019).
\newblock \doi{10.1016/j.apnum.2018.10.007}

\bibitem{lang2017extrapolation}
Lang, J., Hundsdorfer, W.: Extrapolation-based implicit-explicit {P}eer methods
  with optimised stability regions.
\newblock Journal of Computational Physics \textbf{337}, 203--215 (2017)

\bibitem{pareschi2005implicit}
Pareschi, L., Russo, G.: Implicit--explicit {Runge--Kutta} schemes and
  applications to hyperbolic systems with relaxation.
\newblock Journal of Scientific computing \textbf{25}(1), 129--155 (2005)

\bibitem{otp}
Roberts, S., Popov, A.A., Sandu, A.: {ODE} test problems: a {MATLAB} suite of
  initial value problems.
\newblock CoRR \textbf{abs/1901.04098} (2019).
\newblock \urlprefix\url{http://arxiv.org/abs/1901.04098}

\bibitem{Sarshar2019}
Sarshar, A., Roberts, S., Sandu, A.: {Alternating directions implicit
  integration in a general linear method framework}.
\newblock Journal of Computational and Applied Mathematics p. 112619 (2019).
\newblock \doi{10.1016/j.cam.2019.112619}

\bibitem{schneider2018extrapolation}
Schneider, M., Lang, J., Hundsdorfer, W.: Extrapolation-based super-convergent
  implicit-explicit {P}eer methods with {A}-stable implicit part.
\newblock Journal of Computational Physics \textbf{367}, 121--133 (2018)

\bibitem{soleimani2018superconvergent}
Soleimani, B., Weiner, R.: Superconvergent {IMEX} peer methods.
\newblock Applied Numerical Mathematics \textbf{130}, 70--85 (2018)

\bibitem{Sandu_2012_ICCS-IMEX}
Zhang, H., Sandu, A.: A second-order diagonally-implicit-explicit multi-stage
  integration method.
\newblock In: Proceedings of the International Conference on Computational
  Science ICCS 2012, vol.~9, pp. 1039--1046 (2012).
\newblock \doi{10.1016/j.procs.2012.04.112}

\bibitem{Sandu_2014_IMEX-GLM}
Zhang, H., Sandu, A., Blaise, S.: Partitioned and implicit-explicit general
  linear methods for ordinary differential equations.
\newblock Journal of Scientific Computing \textbf{61}(1), 119--144 (2014).
\newblock \doi{10.1007/s10915-014-9819-z}

\bibitem{Sandu_2016_highOrderIMEX-GLM}
Zhang, H., Sandu, A., Blaise, S.: High order implicit--explicit general linear
  methods with optimized stability regions.
\newblock SIAM Journal on Scientific Computing \textbf{38}(3), A1430--A1453
  (2016).
\newblock \doi{10.1137/15M1018897}

\bibitem{Sandu_2015_IMEX-TSRK}
Zharovsky, E., Sandu, A., Zhang, H.: A class of {IMEX} two-step {Runge--Kutta}
  methods.
\newblock SIAM Journal on Numerical Analysis \textbf{53}(1), 321--341 (2015).
\newblock \doi{10.1137/130937883}

\end{thebibliography}

\end{document}